\documentclass[11pt,letterpaper]{article}
\usepackage[utf8]{inputenc}
\usepackage[margin=1.0in]{geometry}
\usepackage{setspace}
\usepackage{amsmath}
\usepackage{amsthm}
\usepackage{amssymb}
\usepackage{mathrsfs}
\usepackage{graphicx}
\usepackage{dsfont}
\usepackage{enumerate}
\usepackage{mathabx}
\usepackage{mathdots}
\usepackage{authblk}
\usepackage{titlesec}
\usepackage{soul}
\usepackage{tikz-cd}
\usepackage{mathtools}
\usepackage{slashed}

\usepackage{hyperref}
\hypersetup{
    colorlinks=true,
    urlcolor=cyan,
    linkcolor=blue,
    citecolor=blue,
}

\usepackage[backend=bibtex,style=numeric]{biblatex}
\newcounter{bigthm}

\newtheoremstyle{citedtheorem}%
  {3pt}
  {3pt}
  {\itshape}
  {}
  {\bfseries}
  {.}
  {.5em}
  {\thmname{#1} \thmnumber{#2} \thmnote{\normalfont#3}}

\newtheoremstyle{citeddefinition}%
  {3pt}
  {3pt}
  {\normalfont}
  {}
  {\bfseries}
  {.}
  {.5em}
  {\thmname{#1} \thmnumber{#2} \thmnote{\normalfont#3}}

\newtheoremstyle{citedexample}%
  {3pt}
  {3pt}
  {\normalfont}
  {}
  {\itshape}
  {.}
  {.5em}
  {\thmname{#1} \thmnumber{#2} \thmnote{\normalfont#3}}

\theoremstyle{plain}
\newtheorem{theorem}{Theorem}[section]
\newtheorem{lemma}[theorem]{Lemma}

\newtheorem{proposition}[theorem]{Proposition}

\theoremstyle{citedtheorem}

\theoremstyle{definition} 
\newtheorem{definition}[theorem]{Definition}

\theoremstyle{citeddefinition}

\theoremstyle{remark}
\newtheorem{remark}[theorem]{Remark}

\theoremstyle{citedexample}

\usepackage{cleveref}

\newcommand{\bP}{\mathbb{P}}

\newcommand{\bS}{\mathbb{S}}

\newcommand{\bW}{\mathbb{W}}

\newcommand{\cC}{\mathcal{C}}

\newcommand{\cE}{\mathcal{E}}
\newcommand{\cF}{\mathcal{F}}

\newcommand{\cK}{\mathcal{K}}

\newcommand{\cM}{\mathcal{M}}
\newcommand{\cN}{\mathcal{N}}

\newcommand{\cR}{\mathcal{R}}

\newcommand{\cT}{\mathcal{T}}

\newcommand{\cV}{\mathcal{V}}
\newcommand{\cW}{\mathcal{W}}

\newcommand{\fF}{\mathfrak{F}}

\newcommand{\fe}{\mathfrak{e}}

\newcommand{\fu}{\mathfrak{u}}

\newcommand{\sC}{\mathscr{C}}

\newcommand{\sE}{\mathscr{E}}

\newcommand{\sG}{\mathscr{G}}

\newcommand{\sN}{\mathscr{N}}
\newcommand{\sO}{\mathscr{O}}

\newcommand{\sU}{\mathscr{U}}
\newcommand{\sV}{\mathscr{V}}

\renewcommand{\a}{\alpha} 
\renewcommand{\b}{\beta} 

\newcommand{\x}{\xi}

\newcommand{\vp}{\varphi}

\renewcommand{\l}{\ell} 

\newcommand{\Z}{\mathbb{Z}}

\newcommand{\R}{\mathbb{R}}
\newcommand{\C}{\mathbb{C}}

\renewcommand{\d}{\partial} 

\newcommand{\sub}{\subseteq}
\newcommand{\setm}{\, \setminus \,}

\newcommand{\8}{\infty}
\newcommand{\<}{\left\langle}
\renewcommand{\>}{\right\rangle}

\DeclareMathOperator{\Hom}{Hom}
\DeclareMathOperator{\End}{End}
\DeclareMathOperator{\id}{id}

\DeclareMathOperator{\rank}{rank}
\DeclareMathOperator{\Tr}{Tr}

\DeclareMathOperator{\Cl}{C\l}

\DeclareMathOperator{\Spin}{Spin}

\DeclareMathOperator{\SO}{SO}

\numberwithin{equation}{section}

\title{\vspace{-20pt}The Calder{\'o}n problem for harmonic spinors}
\author{Carlos Valero}

\begin{document}

\maketitle

\begin{abstract}
    We show that the conformal class of a compact spin manifold with boundary is uniquely determined by its boundary values of harmonic spinors if the metric is locally conformal to a real-analytic metric in some coordinates. More precisely, we consider the boundary conjugation (BC) map for the Dirac operator corresponding to MIT boundary conditions, and prove that if two locally conformally real-analytic (LCRA) metrics have equivalent BC maps over an open subset of the boundary, then the two manifolds are conformally equivalent. The class of LCRA manifolds includes conformally Einstein manifolds, as well as all smooth surfaces. We also show that a real-analytic unitary connection on an LCRA manifold is uniquely determined up to local gauge equivalence by the BC map of its twisted Dirac operator in dimensions greater than $2$.
\end{abstract}

\tableofcontents

\newpage

\section{Introduction}\label{sec: intro}

The geometric (or anisotropic) Calder{\'o}n problem consists in determining a Riemannian manifold with boundary $(M,g)$ from its Dirichlet-to-Neumann (DN) map $\Lambda_g : C^{\8}(\d M) \to C^{\8}(\d M)$, which sends $f \in C^{\8}(\d M)$ to the normal derivative of its harmonic extension to $M$. This problem is a geometric formulation of a question originally posed by Alberto Calder{\'o}n in \cite{Calderon1980}, wherein he asks whether one can determine the conductivity of a material in a region of space from measurements of voltage and current on the boundary. One obtains the geometric version of the problem from the physical one by considering anisotropic conductivities and viewing them as Riemannian metrics. One formulated geometrically, it is clear that one can study the Calder{\'o}n problem not only for the Laplacian, but indeed for many of the fundamental differential operators in Riemannian geometry; see \cite{KLU2011, Cekic2017, LTS2022, EV2024, Valero2025} for example. These geometric Calder{\'o}n problems form part of the wider program of understanding the extent to which underlying geometric objects, such as metrics and connections, are uniquely determined by knowledge of solutions to corresponding equations on the boundary. In this paper, we consider the Calder{\'o}n problem for a fundamental operator in geometry and physics which has hitherto received relatively little attention: the Dirac operator.

In less geometric settings, where the Dirac equation is thought of as an elliptic linear system on domains in $\R^n$, there are a number of works in the literature that consider the inverse problem of recovering a vector or matrix potential from boundary values of solutions to the Dirac equation; see \cite{Tsuchida1998, Li2007, Salo2008, Salo2010} and the references therein, for example. On the other hand, there are significantly fewer works that treat inverse problems for Dirac operators in more general geometric settings. One such work is that of Kurylev and Lassas \cite{KurylevLassas2009}, in which they show that a Riemannian manifold with boundary and Dirac bundle structure can be recovered from the spectral data of a Dirac operator with self-adjoint boundary conditions, where by spectral data we mean the set of eigenvalues and boundary values of the eigenfunctions. Equivalently, instead of the spectral data, one can start with the boundary response operator of the hyperbolic equation $i\d_t + D$ on $M \times \R$. Another work treating inverse problems for Dirac operators in geometric settings is \cite{QuanUhlmann2022}, wherein it is shown that a closed Riemannian manifold is recovered from a local source-to-solution map of the {\em fractional} Dirac operator. We also mention \cite{Roberts2022}, in which an algebraic version of the boundary control method is applied to an inverse boundary value problem for the Hodge--Dirac operator, as well as \cite{Valero2024}, which deals with the Calder{\'o}n problem for unitary connections that couple to (bona fide) spinors via a second-order version of the Yang--Mills--Dirac equations. Finally, we cite \cite{Valero2025}, in which an analogue of the Dirichlet-to-Neumann map, called the boundary conjugation (BC) map, is defined for the twisted Dirac operator $D_A - m$ on a spin manifold with local boundary conditions. It is shown in \cite{Valero2025} that when $m \neq 0$ and the data is analytic, then one can recover a Riemannian manifold and unitary connection from the BC map. The purpose of the present paper is to extend the results of \cite{Valero2025} to the case $m = 0$, which, owing to the nature of the Dirac operator on spinors, possesses an important symmetry not present in the $m \neq 0$ case: conformal covariance.

More precisely, as in \cite{Valero2025}, we consider a spin manifold with boundary $(M,g)$, on which there exists an irreducible spinor bundle $\bS$. Given a Hermitian bundle $(E,\eta)$ over $M$ and unitary connection $A$ on $E$, one can consider the twisted Dirac operator $D_A : \Gamma(\bS \otimes E) \to \Gamma(\bS \otimes E)$. We define the boundary conjugation (BC) map $\Theta_{g,A}$ corresponding to MIT boundary conditions as follows. Let $\gamma(\nu) \in \End(\bS \otimes E)|_{\d M}$ denote Clifford multiplication by the unit normal $\nu$ on the boundary. We have a decomposition $(\bS \otimes E)|_{\d M} = \bW_E^+ \oplus \bW_E^{-}$, where $\bW_E^{\mp}$ is the $\pm i$-eigenbundle of $\gamma(\nu)$. Let $\bP^{\pm}$ denote the orthogonal projection $(\bS \otimes E)|_{\d M} \to \bW_E^{\pm}$. Then, for any $f \in C^{\8}(\d M, \bW_E^+)$, there exists a unique smooth section $\psi_f \in \Gamma(\bS \otimes E)$ on $M$ satisfying
\begin{equation}
    \begin{cases}
        D_A \psi_f = 0, \\
        \bP^+ \psi_f|_{\d M} = f.
    \end{cases}
\end{equation}
We thus define the (MIT) BC map $\Theta_{g,A} : C^{\8}(\d M, \bW_E^+) \to C^{\8}(\d M, \bW_E^-)$ by
\begin{equation}
    \Theta_{g,A}(f) := \bP^{-} \psi_f|_{\d M}.
\end{equation}
Note that if $\sU$ is an open subset of $\d M$, we can restrict $\Theta_{g,A}$ to $f$ that are supported in $\sU$, thus obtaining a {\em local} BC map $\Theta_{g,A,\sU} : C^{\8}(\sU, \bW_E^+) \to C^{\8}(\sU, \bW_E^-)$. We shall denote the local BC map for the untwisted Dirac operator $D : \Gamma(\bS) \to \Gamma(\bS)$ simply by $\Theta_{g,\sU}$. Like the DN map for the Laplacian, the BC map characterizes the Cauchy data for the Dirac operator. Unlike with the DN map, different metrics have different a priori spinor bundles, and we thus need to make sense of how to compare the BC maps arising from different metrics. We make this precise in Definition \ref{def: equivalence}, and when we say that two BC maps are {\em equivalent}, as in the statement of Theorem \ref{main theorem} below, we mean that they are equivalent in the sense of Definition \ref{def: equivalence}.

In order to state our main result, we first introduce locally conformally real-analytic (LCRA) manifolds. An LCRA manifold is a Riemannian manifold $(M,g)$ such that the metric $g$ is locally conformal to a metric that is real-analytic in some coordinates (a more precise definition is given in Section \ref{subsec: LCRA}). In the context of inverse problems, LCRA manifolds were first introduced in \cite{LTS2022} when considering the Calder{\'o}n problem for the conformal Laplacian (though the definition we use is slightly more general; see Section \ref{subsec: LCRA}). The class of LCRA manifolds includes all smooth surfaces, conformally Einstein manifolds, as well as Bach and obstruction flat manifolds, which are conformally invariant generalizations of Einstein manifolds in even dimensions $n \geq 4$ \cite{FG1985, FG2012}. 

In addition to working over LCRA manifolds, when we consider the setting of twisted Dirac operators in dimensions $n \geq 3$, we shall also assume that the auxiliary data $(E, \eta, A)$ by which we are twisting the Dirac operator is admissible (see Definition \ref{def: admissible}); that is to say, that this data is essentially real-analytic up to the boundary in a way that is compatible with the LCRA structure. We can now state our main uniqueness result for the Calder{\'o}n problem for $D_A$:



\begin{theorem}\label{main theorem}
    For $i \in \{1,2\}$, let $(M_i, g_i)$ be a compact $n$-dimensional LCRA spin manifold with boundary, such that the boundaries are identified along an open subset $\sU \sub \d M_i$.
    \begin{enumerate}[{\em \ \ \ (1)}]
        \item Let $n \geq 3$. Suppose that $(E_i, \eta_i, A_i)$ is a Hermitian vector bundle over $M_i$ endowed with a unitary connection $A_i$, which is admissible in the sense of Definition \ref{def: admissible}. If $\Theta_{g_1,A_1,\sU}$ is equivalent to $\Theta_{g_2, A_2, \sU}$, then there exists a complex line bundle $L$ over $M_2$, a flat connection $\nabla^L$ on $L$, and a unitary isomorphism $\Phi_E : E_1 \to E_2 \otimes L$, such that $\Phi_E^*(\nabla^{A_2} \otimes \nabla^L) = \nabla^{A_1}$. Moreover, $\Phi_E$ covers a conformal equivalence $\vp : (M_1, g_1) \to (M_2, g_2)$ with $\vp|_{\sU} = \id$.

        \item Let $n = 2$. If $\Theta_{g_1,\sU}$ is equivalent to $\Theta_{g_2,\sU}$ and $g_1|_{\sU} = \lambda g_2|_{\sU}$ for some constant $\lambda$, then there is a conformal equivalence $\vp : (M_1, g_1) \to (M_2, g_2)$ with $\vp|_{\sU} = \id$.
    \end{enumerate}
\end{theorem}

The conclusion of (1) in Theorem \ref{main theorem} implies that the two connections $A_1$ and $A_2$ are locally gauge equivalent, and can only fail to be globally gauge equivalent in a way that is related to the non-uniqueness of spin structures on a spin manifold (see Example 4.8 in \cite{Valero2025}). Of course, as in \cite{Valero2025}, it follows that if we know a priori that the line bundle $L$ must be trivial, for example if the manifolds satisfy $H^1(M_i, \Z_2) = 0$ or if the isomorphism class of the bundles $E_i$ is known, then $\Phi_E$ is a global gauge equivalence betwen $A_1$ and $A_2$.

The proof of Theorem \ref{main theorem} that we give in this paper is an adaptation of the method used in \cite{LTS2022} for the conformal Laplacian, in which a conformal equivalence is constructed by using conformally rescaled Green's functions to embed both manifolds into a suitable common Sobolev space. This approach is itself an extension of the method used by Lassas, Taylor, and Uhlmann to solve the Calder{\'o}n problem for the Laplacian in the case of analytic data \cite{LassasUhlmann2001, LTU2003}, and which has also been used to much success in tackling the Calder{\'o}n problem for various other geometric operators in the presence of analytic data \cite{KLU2011, LTS2022, EV2024, Gabdurakhmanov2025, Valero2025}. The version of the argument that we present in this paper, however, is somewhat distinct from that given in \cite{LTS2022}. Unlike in \cite{LTS2022}, our proof is based on finding a suitable conformal factor in analytic coordinates near the boundary, so that after rescaling by said factor, the metrics can be analytically extended beyond the boundary, to a region in which they are isometric. This approach has a number of advantages. First, the local nature of this argument allows us to deal with the case of partial data, where the BC map is known only on an open subset $\sU \sub \d M$, rather than restricting to global data as in \cite{LTS2022}. Second, our use of the conformal symmetry of the problem means that we can require less regularity of the metric on the boundary; indeed, our definition of LCRA manifolds is slightly weaker than the one given in \cite{LTS2022} (see Remark \ref{remark: LCRA definition is more general}). Finally, we note that our method of finding such a conformal factor near the boundary relies on solving a local Cauchy problem, and appears easier to generalize to other conformal invariant operators, including the conformal Laplacian.



The structure of the paper is as follows. In Section \ref{sec: Preliminaries}, we review some of the basic definitions and notions needed in the statement and proof of Theorem \ref{main theorem}. In particular, Section \ref{subsec: spinors} covers the basics of spinors and Dirac operators; Section \ref{subsec: MIT BC map} introduces the MIT boundary value problem for Dirac operators on bounded domains along with the corresponding BC map; and Section \ref{subsec: LCRA} formally introduces locally conformally real-analytic (LCRA) manifolds and discusses some of their important properties. In Section \ref{sec: proof full section} we give the proof of Theorem \ref{main theorem}, which is broken into 3 parts: in Section \ref{sec: extension}, we extend our manifolds in a real-analytic fashion after a suitable conformal rescaling; in Section \ref{sec: Greens functions determined} we show that the BC map determines a conformally rescaled version of the Green's kernels on the extended manifold; and finally in Section \ref{sec: reconstruction} we explicitly construct the isomorphism and conformal equivalence of Theorem \ref{main theorem} by using the rescaled Green's kernels to embed our bundles into a common Sobolev space.




\vspace{10pt}

\section{Preliminaries}\label{sec: Preliminaries}

In this section, we briefly review some of the basic definitions and results necessary to understand the statement and proof of Theorem \ref{main theorem}. In Section \ref{subsec: spinors} we give an overview spinors and Dirac operators, as well as some of their important properties. We also define local boundary conditions for the Dirac operator and define the associated boundary conjugation map $\Theta_{g,A,\sU}$. In Section \ref{subsec: LCRA} we define locally conformally real-analytic manifolds, and discuss some of their properties.

\vspace{10pt}
\subsection{Spinors and twisted Dirac operators}\label{subsec: spinors}

We first recall some fundamental ideas from the theory of Clifford algebras, spin representations, and spinor bundles. For details, we refer the reader to Chapters I and II of \cite{LawsonMichelsohn}.

\begin{definition}
    Let $(V,g)$ be a real inner product space. The {\em Clifford algebra} $\Cl(V,g)$ is the real associative algebra with unit $1$ generated by all $v \in V$ subject to the relations 
    \begin{equation}
        v_1 v_2 + v_2 v_1 = -2 g(v_1, v_2)
    \end{equation}
    for $v_1, v_2 \in V$. We let $\Cl_n := \Cl(\R^n, g_n)$ where $g_n$ is the standard inner product on $\R^n$. Of course, every Clifford algebra $\Cl(V)$ is non-canonically isomorphic to $\Cl_n$ where $n = \dim{V}$.
\end{definition}

Note that we have $V \sub \Cl(V,g)$, and $v^2 = -|v|^2$ for all $v \in V$. Moreover, if $\dim{V} = n$, then $\Cl(V,g)$ is also generated by an orthonormal basis $e_1, \dots, e_n$ of $V$ subject to the relations
\begin{equation}
    e_i^2 = -1,
\end{equation}
\begin{equation}
    e_i e_j + e_j e_i = 0, \ \ i \neq j.
\end{equation}
Note also that $\Cl_n$ has dimension $2^n$ and is isomorphic as a vector space to the exterior algebra.

\vspace{5pt}
\begin{remark}\label{remark: spin group}
    The {\em Spin group} $\Spin(n) \sub \Cl_n$ is defined to be the group multiplicatively generated by all elements of the form $v_1 \cdots v_{2k}$ where $|v_i| = 1$. The significance of the spin group lies in the fact that one can define a two-to-one Lie group covering $\lambda : \Spin(n) \to SO(n)$. In particular, for $n \geq 3$, the group $\Spin(n)$ is simply connected and isomorphic to the universal cover of $SO(n)$.
\end{remark}
\vspace{5pt}

A (complex) {\em representation} of $\Cl(V,g)$ is a {\em real} homomorphism $F : \Cl(V,g) \to \End{W}$, where $W$ is a finite-dimensional complex vector space. An irreducible representation of $\Cl_n$ is called a {\em spinor representation}. There is a unique spinor representation modulo equivalence when $n$ is even, which we denote by $\Delta_n$, and two inequivalent ones when $n$ is odd, denoted by $\Delta_n^{\pm}$. The spinor representations have dimension $2^{\left[ \frac{n}{2} \right]}$. In fact, the Clifford algebra $\Cl_n$ is isomorphic to $\End{\Delta_n}$ when $n$ is even, and $\End{\Delta_n^+} \oplus \End{\Delta_n^{-}}$ when $n$ is odd. We also note that it always possible to equip a representation $W$ of $\Cl_n$ with a Hermitian inner product such that multiplication by unit vectors is unitary. We shall assume such a Hermitian inner product has been chosen.

\vspace{5pt}
\begin{remark}\label{remark: spin representations for spin group}
    The {\em spinor representations} $\Delta^{\Spin}_n$ of $\Spin(n)$ are defined by restricting the spinor representations of $\Cl_n$ to $\Spin(n) \sub \Cl_n$. When $n$ is odd, restricting either of $\Delta_n^{\pm}$ to $\Spin(n)$ yields equivalent representations of $\Spin(n)$. When $n$ is even, restricting $\Delta_n$ to $\Spin(n)$ yields the spinor representation for $\Spin(n)$, which is reducible: we have $\Delta^{\Spin}_n = \Delta^{\Spin, +}_n \oplus \Delta^{\Spin, -}_n$, where $\Delta^{\Spin, \pm}_n$ is an irreducible representation for $\Spin(n)$. We mention this here as it is more common to consider spinor representations from the point of view of $\Spin(n)$.
\end{remark}
\vspace{5pt}

We now want to take the above linear-algebraic setting and transplant it onto a manifold:

\begin{definition}\label{def: spinor bundles}
    Let $(M,g)$ be an oriented Riemannian manifold. The {\em Clifford bundle} $\Cl(M,g)$ is defined as the bundle of real associative algebras whose fiber at $x \in M$ is the Clifford algebra $\Cl(T_xM, g_x)$. We say that $(M,g)$ is a {\em spin manifold} if there exists on $M$ a complex vector bundle $\bS$, called a {\em spinor bundle}, and a smooth map $\gamma : \Cl(M,g) \to \End(\bS)$ such that for each $x \in M$, we have that $\gamma_x : \Cl(M,g)_x \to \End{\bS_x}$ is isomorphic to a fixed spinor representation of $\Cl_n$.
\end{definition}

 Note that a map $\gamma : \Cl(M,g) \to \End(\bS)$ as defined above is uniquely determined by its values on $TM \sub \Cl(M,g)$. In particular, the specification of such a map is essentially equivalent to a giving a map $\gamma : TM \to \End(\bS)$, denoted by the same symbol, such that
 \begin{equation}
     \gamma(X)\gamma(Y) + \gamma(Y)\gamma(X) = -2g(X,Y) \id_{\bS}.
 \end{equation}
 We refer to $\gamma$ as {\em Clifford multiplication,} and shall sometimes write $\gamma(X)\psi$ as $X \cdot \psi$ for $X \in T_xM$ and $\psi \in \bS_x$. Due to the isomorphism $TM \to T^*M$ induced by the metric, we also have an associated map $\gamma : T^*M \to \End(\bS)$, which we shall also denote by the same symbol. We also mention that it is always possible to equip $\bS$ with a Hermitian metric $\< \cdot, \cdot \>$ such that
\begin{equation}
    \< \gamma(e) \psi_1, \psi_2 \> = -\< \psi_1, \gamma(e) \psi_2 \>
\end{equation}
for all $e \in TM$ with $|e| = 1$. In other words, Clifford multiplication by unit vectors is unitary. We will assume henceforward that our spinor bundles are equipped with such a Hermitian metric. 

\vspace{5pt}
\begin{remark}\label{remark: spin structure construction}
    A spinor bundle as defined in Definition \ref{def: spinor bundles} does not necessarily exist for any oriented manifold, and when it does it is not necessarily unique. In fact, the obstruction to the existence of such a bundle is topological: a spinor bundle exists if and only if the second Stieffel--Whitney class $w_2(M) \in H^2(M, \Z_2)$ vanishes. On such manifolds, it is possible to find a principal $\Spin(n)$-bundle $P_{\Spin}$ and a double-covering map $\xi$ from $P_{\Spin(n)}$ to the orthonormal frame bundle $P_{\SO}$ of the metric $g$, such that $\xi$ intertwines with the double-covering $\lambda : \Spin(n) \to SO(n)$ from Remark \ref{remark: spin group}. The pair  $(P_{\Spin}, \x)$ is called a {\em spin structure}. The set of inequivalent spin structures is in one-to-one correspondence with $H^1(M,\Z_2)$. Given a choice of spin structure $P_{\Spin}$, the spinor bundle $\bS$ can be defined as the associated bundle $P_{\Spin} \times_{\rho} \Delta_n^{\Spin}$, where $\rho : \Spin(n) \to GL(\Delta_n^{\Spin})$ is the spinor representation of $\Spin(n)$. In particular, we note that given a local orthonormal frame $(e_i)$, which is a local section of $P_{\SO}$, there are two local orthonormal frames for $\bS$ determined by $(e_i)$. These are denoted $(\sigma_{\a})$ and $-(\sigma_{\a})$, and satisfy $\xi(\pm (\sigma_{\a})) = (e_i)$, where $\xi : P_{\Spin} \to P_{\SO}$ is as above and we think of $P_{\Spin}$ as lying in the orthonormal frame bundle of the associated bundle $\bS$. We shall refer to $\pm (\sigma_{\a})$ as the {\em canonical spinor frames} determined by $(e_i)$.
\end{remark}
\vspace{5pt}

The spinor bundle $\bS$ carries a special connection $\nabla^s$ called the {\em spin connection}, which is induced by the Levi-Civita connection as follows. Let $(e_i)$ be a local orthonormal frame on $U \subseteq M$ with respect to which the Levi-Civita connection $1$-form $\omega^i_{\ j} \in \Omega^1(U, \mathfrak{so}(n))$ is given by
\begin{equation}
    \nabla_X e_j = \omega^i_{\ j}(X) e_i \label{levi civita one form}
\end{equation}
for $X \in TM|_U$, where we employ the summation convention. In a canonical spinor frame determined by $(e_i)$ (see Remark \ref{remark: spin structure construction}), the spin connection $\nabla^s$ on $\bS$ is locally given by
\begin{equation}
    \nabla^s_X \sigma_{\a} = -\frac{1}{2} \sum_{i < j} \omega^i_{\ j}(X) \gamma(e_i) \gamma(e_j) \sigma_{\a},
\end{equation}
where $\omega^i_{\ j}$ is the connection $1$-form determined by $(e_i)$ as in \eqref{levi civita one form}. Here we have explicitly written out the sum, which we shall do on occasion in the interest of clarity. Note that the above definition of the spin connection does not depend on the choice of canonical spinor frame associated with $(e_i)$. The spin connection $\nabla^s$ has the property that it is compatible with Clifford multiplication and the Levi-Civita connection on $TM$: that is, for $X,Y \in \Gamma(TM)$ and $\psi \in \Gamma(\bS)$ we have
\begin{equation}
    \nabla^s_X \left( \gamma(Y)\psi \right) = \gamma(\nabla_X Y) \psi + \gamma(Y) \nabla^s \psi.
\end{equation}
The spin connection on $\bS$ allows us to define the Dirac operator as follows:

\begin{definition}\label{def: dirac operator}
    The {\em Dirac operator} $D : \Gamma(\bS) \to \Gamma(\bS)$ is the first-order differential operator defined as follows: for any local orthonormal frame $(e_i)$ over $U \sub M$, 
    \begin{equation}
        D \psi := \sum_{i=1}^n \gamma(e_i) \nabla^s_{e_i} \psi
    \end{equation}
    for $\psi \in \Gamma(\bS|_U)$. This definition does not depend on the choice of orthonormal frame.
\end{definition}

In fact, we want to consider a more general setting, in which we twist the Dirac operator acting on spinors by an auxiliary bundle and unitary connection. Let us thus assume that $(E, \eta)$ is a Hermitian vector bundle over $M$, and let $\nabla^A$ be a unitary connection on $E$. Then we can endow the bundle $\bS_E := \bS \otimes E$ with the connection $\nabla^s \otimes \nabla^A$, which we shall continue to denote by $\nabla^A$ when there is no confusion. Note also that Clifford multiplication extends to $\bS_E$ by acting only on the first factor. We continue to denote this map by $\gamma$. Moreover, $\bS_E$ inherits a Hermitian inner product from $\eta$ and the one on $\bS$. We may thus define the {\em twisted Dirac operator} $D_A$, which acts on sections of $\bS_E$, by formally replacing $\nabla^s$ by $\nabla^A$ in Definition \ref{def: dirac operator}. 

\vspace{5pt}
\begin{remark}
    We will often simply use the symbol $A$ to denote both the unitary connection $\nabla^A$ on $E$ as well as its local representative, a $\fu(N)$-valued $1$-form, in a local trivialization.
\end{remark}
\vspace{5pt}

The twisted Dirac operator $D_A$ has principal symbol $\sigma(D_A)(x,\x) = i \gamma(\x)$. In particular, $D_A$ is elliptic, and its square $D_A^2$ has scalar principal symbol given by $|\x|^2$. In fact, we have the following formula for $D_A^2$, known as the Lichnerowicz--Weitzenb{\"o}ck formula:
\begin{equation}\label{Lichnerowicz formula}
    D_A^2 = \left( \nabla^A \right)^* \nabla^A + \frac{1}{4} S_g + \fF_A,
\end{equation}
where $\left( \nabla^A \right)^*$ denotes the formal adjoint of $\nabla^A$, $S_g$ denotes the scalar curvature of the metric $g$, and the Weitzenb{\"o}ck curvature operator $\fF_A \in \End(\bS_E)$ is defined by
\begin{equation}
    \fF_A := \sum_{j < k} \gamma(e_j) \gamma(e_k) F_A(e_j, e_k),
\end{equation}
where $F_A \in \Omega^2(\End{E})$ is the curvature $2$-form of the connection $A$.

We note that all of the above definitions carry over straightforwardly to spin manifolds with boundary, which one can define as smooth domains in spin manifolds without boundary. In this setting, we have the following Green's identity for the twisted Dirac operator \cite[\S II.5]{LawsonMichelsohn}:
\begin{equation}\label{Green's identity for D_A}
    \int_M \<  D_A  \psi_1, \psi_2 \> - \<  \psi_1,  D_A \psi_2 \> = -\int_{\d M} \< \gamma(\nu) \psi_1, \psi_2 \>
\end{equation}
for $\psi_1, \psi_2 \in \Gamma(\bS_E)$, where $\nu$ is the inner unit normal to $\d M$. In particular, for a unitary connection $A$, the twisted Dirac operator $D_A$ is formally self-adjoint.

Finally, we recall how the Dirac operator behaves under conformal transformations, which shall play an important role in the present paper. The main subtlety with comparing Dirac operators corresponding to different metrics is that the spinor bundle is defined with respect to a given metric, and there is not an obvious way of identifying them. When the two metrics are conformal, however, there a canonical identification, which we recall here. We focus on the untwisted Dirac operator on $\bS$, though everything carries over to the twisted Dirac operator on $\bS_E$, since conformal scaling does not affect the $E$ factor. For details on the following, see \cite{Hijazi1986}.

Recall that a metric $\bar{g}$ on $M$ is conformal to another metric $g$ if $\bar{g} = e^{2u} g$ for some smooth function $u \in C^{\8}(M)$. In this case, there is a canonical map from the orthonormal frame bundle of $g$ to that of $\bar{g}$ obtained by sending an $g$-orthonormal frame $(e_i)$ to the $\bar{g}$-orthonormal frame $(\bar{e}_i)$, where $\bar{e}_i := e^{-u} e_i$. This induces an isomorphism between corresponding spin structures, and hence between the corresponding spinors bundles $\bS \to \bar{\bS}$. We shall denote this canonical isomorphism simply by $\psi \mapsto \bar{\psi}$. Once we identify the spinor bundles of $g$ and $\bar{g}$ via this isomorphism, we can compare the Dirac operators $D_{g}$ and $D_{\bar{g}}$. The Clifford multiplication maps of $g$ and $\bar{g}$ are related by $\bar{\gamma}(X) = e^u \gamma(X)$ for $X \in TM$, while the spin connections satisfy
\begin{equation}\label{spin connection under conformal change}
    \overline{\nabla}^s_X \bar{\psi} = \overline{\nabla^s_X \psi} - \frac{1}{2} \bar{\gamma}(X) \bar{\gamma}(du) \bar{\psi} - \frac{1}{2} X(u) \bar{\psi}.
\end{equation}
As a result, one finds that the Dirac operators $D_g$ and $D_{\bar{g}}$ are related by
\begin{equation}\label{Dirac operator under conformal change v1}
    D_{\bar{g}} \bar{\psi} = e^{-u} \left( \overline{D_g \psi} + \frac{n-1}{2} \bar{\gamma}(du) \bar{\psi} \right),
\end{equation}
where $n$ is the dimension of $M$. Note that the above is equivalent to
\begin{equation}\label{Dirac operator under conformal change exp version}
    \overline{D_g \psi} = e^{\frac{n+1}{2} u} D_{\bar{g}} \left( e^{- \frac{n-1}{2} u} \bar{\psi} \right).
\end{equation}
Recall that we call a spinor $\psi$ {\em harmonic} if $D \psi = 0$. Note that \eqref{Dirac operator under conformal change exp version} implies that if $\psi$ is harmonic for the metric $g$, then $e^{-\frac{n-1}{2} u} \bar{\psi}$ is harmonic for the metric $\bar{g} = e^{2u} g$. As remarked above, under the induced isomorphism $\bS_E \to \bar{\bS}_E$, equation \eqref{Dirac operator under conformal change exp version} continues to hold with $D$ replaced by $D_A$.

\vspace{10pt}
\subsection{MIT boundary conditions and the boundary conjugation map}\label{subsec: MIT BC map}

Let us now consider boundary conditions for the Dirac operator on a spin manifold with boundary. We begin our discussion by focusing on the untwisted Dirac operator for simplicity, but all that we consider here applies equally well to twisted Dirac operators. 

The most common boundary conditions for the Dirac operator considered in the literature are the fully non-local APS boundary conditions introduced in \cite{APS1975I, APS1975II, APS1976III}, which are defined by a $0$-th order pseudodifferential $L^2$-projector on the boundary, and play an important role in physics and geometry. However, since we are interested in the Calder{\'o}n problem not only for global data, but also for partial data on some subset $\sU \sub \d M$, it is more convenient for us to consider local boundary conditions for $D$. The most common {\em local} boundary conditions considered for the Dirac operator are the chiral boundary conditions (see \cite{HMZ2001, HMR2002, Raulot2011, Valero2025}, for example), and the MIT boundary conditions; the latter were introduced in \cite{Johnson1975} to model the confinement of quarks. While the chiral boundary conditions have the advantage of being self-adjoint for the Dirac operator, the MIT boundary conditions have the notable advantage that they can be defined on any spin manifold with boundary, unlike chiral boundary conditions, which are subject to the existence of a chirality operator (see \cite{HMZ2001, HMR2002}). We shall therefore focus on the MIT conditions in this paper.

Let $M$ be a compact spin manifold with boundary $\d M$, endowed with a spinor bundle $\bS$. Let $\nu$ be the inward unit normal defined on the boundary. Note that $i \gamma(\nu)$ is an involution on $\bS|_{\d M}$. In particular, $\bS|_{\d M}$ decomposes as the direct sum $\bW^+ \oplus \bW^-$, where $\bW^{\pm}$ is the $\pm 1$-eigenbundle of the endomorphism $i \gamma(\nu)$. Let $\bP^{\pm} : \bS|_{\d M} \to \bW^{\pm}$ be the corresponding orthogonal projectors:
\begin{equation}
    \bP^{\pm} := \frac{1}{2} \left( 1 \pm i \gamma(\nu) \right).
\end{equation}
Then $\bP^{\pm}$ defines an elliptic boundary condition for the Dirac operator. In particular, the restriction of $D$ to the space of all $\psi \in H^1(M, \bS)$ such that $\bP^{\pm} \psi = 0$ is a Fredholm operator whose spectrum consists entirely of isolated complex eigenvalues $\lambda_k^{\pm}$ of finite multiplicity satisfying $|\lambda_k^{\pm}| \to \infty$ as $k \to \infty$ \cite{HMZ2001, HMR2002}. We call the set of eigenvalues $\lambda_k^{\pm}$ the MIT$^\pm$ spectrum. We have the following \cite{HMR2002}: 

\begin{proposition}
    The MIT$^\pm$ spectrum lies entirely in $\C \setm \R$. In particular, the Dirac operator with MIT boundary conditions is invertible. 
\end{proposition}

\begin{proof}
    Indeed, if $D \psi = m \psi$ with $m \in \R$, and $i \gamma(\nu) \psi = \pm \psi$, then \eqref{Green's identity for D_A} gives
    \begin{equation}
        0 = i \int_{\d M} |\psi|^2,
    \end{equation}
    which implies that $\psi|_{\d M} = 0$. But since Dirac operators satisfy a unique continuation principle from the boundary \cite{BW1993}, it follows that we must have $\psi = 0$ on $M$. 
\end{proof}

\begin{remark}
    The above discussion carries over to the setting of twisted Dirac operators $D_A$ when $A$ is a unitary connection. In this case, we have $\bS_E := \bW^+_E \oplus \bW^-_E$, where $\bW^{\pm}_E := \bW^{\pm} \otimes E$. As with Clifford multiplication, the projectors $\bP^{\pm}$ act only on the $\bS$ factor.
\end{remark}
\vspace{5pt}

Since $D_A$ with MIT boundary conditions is invertible, it follows that we may uniquely solve
\begin{equation}\label{MIT BVP}
    \begin{cases}
        D_A \psi = 0, \\
        \bP^+ \psi = f,
    \end{cases}
\end{equation}
for any $f \in C^\8(\d M, \bW_E^+)$. This motivates the following definition:

\begin{definition}\label{def: BC map}
    Let $(M,g)$ be a compact spin manifold with boundary and $(E,h,A)$ a Hermitian bundle over $M$ with unitary connection as above. We define the {\em boundary conjugation (BC)} map $\Theta_{g,A} : C^{\8}(\d M, \bW_E^+) \to C^{\8}(\d M, \bW_E^-)$ by
    \begin{equation}
        \Theta_{g,A}(f) := \bP^{-} \psi|_{\d M},
    \end{equation}
    where $\psi$ is the unique solution to \eqref{MIT BVP} and $\bP^- : \bS_E|_{\d M} \to \bW_E^{-}$ is the orthogonal projector defined above. If $\sU$ is an open subset of the boundary, we define the {\em local} BC map $\Theta_{g,A,\sU}$ by restricting $\Theta_{g,A}$ to smooth sections of $\bW_E^{+}$ that have compact support in $\sU$.
\end{definition}

\begin{remark}
    For simplicity, we shall often assume that $\sU \sub \d M$ is connected.
\end{remark}

The boundary conjugation map is to the Dirac operator what the Dirichlet-to-Neumann map is for the Laplacian, and like the Dirichlet-to-Neumann map, the BC map characterizes the Cauchy data of harmonic spinors on $M$, which, since $D_A$ is a first-order operator, consists of their boundary values. As shown in \cite{Valero2025} for chiral boundary conditions, for real $m$ not in the MIT$^+$ spectrum, one can relate the BC map for $D_A - m$ to the Dirichlet-to-Neumann map for $D^2_A - m^2$, so long as the latter is defined; that is, when $m^2$ is not in the Dirichlet spectrum of $D_A^2$. As we have seen, $0$ is never in the MIT spectrum. In fact this is also true for the Dirichlet spectrum of $D_A^2$:

\begin{proposition}\label{prop: D_A^2 invertible}
    The Dirac Laplacian $D_A^2$ with Dirichlet boundary conditions is invertible.
\end{proposition}

\begin{proof}
    Suppose $D_A^2 \psi = 0$ and $\psi|_{\d M} = 0$. Then, by the Green's identity \eqref{Green's identity for D_A}, we have
    \begin{equation}
        \int_{M} |D_A \psi|^2 = 0,
    \end{equation}
    which implies that $D_A \psi = 0$ on $M$. But we also have $\psi|_{\d M} = 0$. Therefore, it follows from the unique continuation principle \cite{BW1993} that $\psi = 0$ on $M$.
\end{proof}

Since $0$ is not in the Dirichlet spectrum of $D_A^2$, the Dirichlet-to-Neumann map for $D_A^2$ can be defined, and by relating this DN map to the BC map as in \cite{Valero2025}, one can compute the full symbol of the BC map in a local trivialization using the same method. In fact, the symbol of the MIT BC map is formally identical to that of the chiral BC map considered in \cite{Valero2025}, replacing all instances of $\gamma(\nu)$ by $\pm i$ whenever it acts on $\bW^{\mp}$. Hence, one finds in particular the following analogue of \cite[Proposition 3.4]{Valero2025} for the MIT boundary conjugation map: 

\begin{proposition}\label{prop: principal symbol of BC map}
    The BC map is a $0$-th order pseudodifferential operator with principal symbol
    \begin{equation}
        \vartheta_0(x, \x) = \frac{\gamma(\x)}{|\x|}.
    \end{equation}
\end{proposition}

\vspace{5pt}
\begin{remark}\label{remark: principal symbol determines conformal class}
    The principal symbol $\vartheta_0(x,\x)$ determines the conformal class of the pullback metric $g|_{\d M}$ on the boundary. Indeed, by the Clifford relations, we have for $\x,\eta \in T_x^*(\d M) \setm 0$
\begin{equation}\label{theta_0 determines conformal class}
    \vartheta_0(x,\x)\vartheta_0(x,\eta) +  \vartheta_0(x,\eta)\vartheta_0(x,\x) = - \frac{2 g(\x,\eta)}{|\x||\eta|}.
\end{equation}
The right-hand-side of \eqref{theta_0 determines conformal class} is the cosine of the angle between the rays determined by $\x$ and $\eta$. This implies that $\vartheta_0$ uniquely determines the conformal class of $g|_{\d M}$. On the other hand, by the discussion preceding \eqref{spin connection under conformal change}, one can see that $\vartheta_0$ is invariant under conformal scaling.
\end{remark}
\vspace{5pt}

\begin{remark}
    One could of course, instead of the ``plus-to-minus" BC map defined above, consider the opposite ``minus-to-plus" BC map $\Theta_{g,A,\sU}^{-} : C^{\8}(\sU, \bW_E^{-}) \to C^{\8}(\sU, \bW_E^+)$, defined in complete analogy with Definition \ref{def: BC map}, interchanging the roles of $\bP^{\pm}$. It is not difficult to see that these two maps are inverses of one another. Moreover, Proposition \ref{prop: principal symbol of BC map} implies that the BC maps extend to bounded maps $L^2(\sU, \bW_E^{\pm}) \to L^2(\sU, \bW_E^{\mp})$. Using the Green's identity \eqref{Green's identity for D_A}, it is easy to see that $\Theta_{g,A,
    \sU}^{-}$ must in fact also be the adjoint of $\Theta_{g,A,\sU}$. It therefore follows that the BC maps are in fact {\em unitary} isomorphisms $L^2(\sU, \bW_E^{\pm}) \to L^2(\sU, \bW_E^{\mp})$.
\end{remark}
\vspace{5pt}


We now want to consider the question of how to compare BC maps arising from different metrics and connections. Let us first consider a boundary-preserving isometry $\vp : (M, g) \to (M',g')$ such that $\vp|_{\sU} = \id$. Such a map induces an isomorphism between the oriented orthonormal frame bundles of $g$ and $g'$, which lifts to a map between corresponding spin structures. We thus obtain an induced map $\hat{\vp} : \bS \to \bS'$ between the corresponding spinor bundles. It is easy to see that twisted Dirac operators satisfy $D_{g',A'} = \hat{\vp} \circ D_{g,A} \circ \hat{\vp}^{-1}$, where $(E', \eta', A')$ is the pullback of $(E,\eta,A)$ by $\vp$. Note that we also have $\hat{\vp} \circ \bP^{\pm} \circ \hat{\vp}^{-1}$. Therefore, the BC maps satisfy
\begin{equation}
    \Theta_{g',A',\sU} = \hat{\vp} \circ \Theta_{g,A,\sU} \circ \hat{\vp}^{-1}.
\end{equation}
A similar result holds for {\em gauge transformations}, by which we simply mean unitary automorphisms of $(E,\eta,A)$. More precisely, if $\Phi : (E,\eta,A) \to (E',\eta',A')$ is a unitary isomorphism between two Hermitian bundles with unitary connection covering $\id_M$, then we have
\begin{equation}
    \Theta_{g',A',\sU} = \Phi|_{\sU} \circ \Theta_{g,A,\sU} \circ \Phi^{-1}|_{\sU}.
\end{equation}

The above tells us how the BC map behaves under the natural symmetries of the structures involved, namely isometries and gauge transformations. In order to state the Calder{\'o}n problem for the Dirac operator, however, we need a way of comparing the BC maps of two a priori unrelated metrics on $\sU$. Motivated by the above discussion on how the BC map changes under isometries and gauge transformations, we make the following definition: 


\begin{definition}\label{def: equivalence}
    Suppose $(M, g)$ and $(M', g')$ are compact spin manifolds with boundary whose boundaries are identified along $\sU \sub \d M, \d M'$. Let $E$ and $E'$ be Hermitian vector bundles over $M$ and $M'$, respectively. Then we say that $\Theta_{g,A,\sU}$ and $\Theta_{g', A', \sU}$ are {\em (canonically) equivalent} if there is a unitary isomorphism $\chi : E|_{\sU} \to E'|_{\sU}$, and a map $\hat{F} : \bS|_{\sU} \mapsto \bS'|_{\sU}$ induced by an isomorphism of the oriented orthonormal frame bundles of $g$ and $g'$ over $\sU$, such that
    \begin{equation}\label{definition of eqivalence eqn}
        \Theta_{g', A', \sU} = \left( \hat{F} \otimes \chi \right) \circ \Theta_{g, A, \sU} \circ \left( \hat{F} \otimes \chi \right)^{-1}.
    \end{equation}
    We shall write $\Theta_{g, A, \sU} \sim \Theta_{g', A', \sU}$ to mean that $\Theta_{g, A, \sU}$ is (canonically) equivalent to $\Theta_{g', A', \sU}$.
\end{definition}

\vspace{3pt}
\begin{remark}
    We call the above definition {\em canonical} equivalence only to distinguish it from {\em linear} equivalence that we introduce below in Remark \ref{remark: linear equivalence}. Throughout this paper, we shall use the term equivalence without qualifiers to mean canonical equivalence.
\end{remark}
\vspace{5pt}

The discussion preceding Definition \ref{def: equivalence} thus means that if $g$ and $g'$ are related by an isometry, and if $A$ and $A'$ are related by a gauge equivalence, then $\Theta_{g,A,\sU} \sim \Theta_{g', A', \sU}$. Another way to think of equivalence is the following: two BC maps are equivalent if one can choose orthonormal frames for $E|_{\sU}$ and $E'|_{\sU}$, along with orthonormal frames for $g$ and $g'$ over $\sU$, such that with respect to the induced canonical spinor frames, the matrix representatives of $\Theta_{g,A,\sU}$ and $\Theta_{g', A', \sU}$ are equal.


Finally, we want to consider how the BC map behaves under conformal scaling of the metric. Recall from the discussion preceding \eqref{spin connection under conformal change} that if $\bar{g} = e^{2u} g$ for some $u \in C^{\8}(M)$, then we have a canonical identification $\bS \to \bar{\bS}$ of the spinor bundles, which is induced by by an isomorphism of the orthonormal frame bundles, and which we denote simply by $\psi \mapsto \bar{\psi}$. Note that under this identification, we have $\bar{\gamma}(\bar{\nu}) = \gamma(\nu)$. In particular, the MIT boundary projectors $\bP^{\pm}$ are conformally invariant. Using this, as well as the conformal covariance of the Dirac operator \eqref{Dirac operator under conformal change exp version}, one finds that the boundary conjugation maps of $g$ and $\bar{g}$ are related by
\begin{equation}\label{conformal change of BC map}
    \overline{{\Theta}_{g,A,\sU}(f)} = e^{\frac{n-1}{2} u} \Theta_{\bar{g}, A, \sU} \left( e^{-\frac{n-1}{2} u} \bar{f} \right).
\end{equation}
We extend the notation introduced in Definition \ref{def: equivalence} in an obvious way by writing
\begin{equation}\label{conformal change of BC map equivalence}
    \Theta_{g,A,\sU} \sim e^{\frac{n-1}{2} u} \Theta_{\bar{g}, A, \sU} e^{- \frac{n-1}{2} u}.
\end{equation}
Note in particular that only $u|_{\sU}$ appears in the above, and moreover, if $u|_{\sU}$ is constant then the BC map of $\bar{g}$ is equivalent to the that of $g$. This conformal gauge symmetry of the boundary conjugation map shall play a crucial role in the proof of Theorem \ref{main theorem}. 

\vspace{5pt}
\begin{remark}\label{remark: linear equivalence}
    The notion of equivalence (what we call canonical equivalence) given in Definition \ref{def: equivalence} is motivated by the usual way that spinor bundles arising from different metrics are compared in the literature (see \cite{Maier1997} for a more detailed exposition); one typically wants to keep track of the unique structure of spinor bundles as objects arising from the underlying geometry. As mentioned above, \eqref{conformal change of BC map equivalence}, implies that this notion of equivalence distinguishes BC maps of conformally related metrics unless the conformal factor is constant on the boundary. One could instead ask for a weaker notion of equivalence, such as conjugation by a linear isomorphism. We will say that two BC maps $\Theta_{g,A,\sU}$ and $\Theta_{g',A',\sU}$ are {\em linearly equivalent} if \eqref{definition of eqivalence eqn} holds with $\hat{F} : \bS|_{\sU} \to \bS'|_{\sU}$ some arbitrary linear isomorphism, not necessarily induced by a map of orthonormal frame bundles. Note that by \eqref{conformal change of BC map}, two conformally related metrics always have linearly equivalent BC maps, regardless of what values the conformal factor takes on the boundary. Indeed, the required linear isomorphism in this case is provided by $e^{\frac{n-1}{2} u}$ times a canonical equivalence. While the notion of linear equivalence seems a priori to be a much more general relation, it turns out that this conformal scaling is the only thing this notion adds, as the following proposition shows.
\end{remark}

\begin{proposition}\label{prop: linear vs canonical equivalence}
        Suppose that $\Theta_{g,A,\sU}$ and $\Theta_{g',A',\sU}$ are linearly equivalent via a vector bundle isomorphism $F : \bS|_{\sU} \to \bS'|_{\sU}$. Then there is an isomorphism $\hat{Q} : \bS|_{\sU} \to \bS'|_{\sU}$ induced by a map of orthonormal frame bundles over $\sU$ and a smooth non-vanishing function $f \in C^{\8}(\sU)$ such that
        \begin{equation}
            F = f \hat{Q}.
        \end{equation}
        Moreover, $g'|_{\sU} = e^{2u} g|_{\sU}$ for some $u \in C^{\8}(\sU)$. If $n \geq 3$, then $f = e^{-\frac{n-1}{2} u}$ up to a constant.
\end{proposition}

We defer the proof of Proposition \ref{prop: linear vs canonical equivalence} to Appendix \ref{sec: appendix L}. Throughout this paper, we will mostly stick with the more usual notion of canonical equivalence, though we refer to Remark \ref{remark: linear equivalence extension result} below for an example of how some results are changed if one uses the notion of linear equivalence instead of canonical equivalence to compare BC maps, especially in dimension $2$.


\vspace{10pt}
\subsection{Locally conformally real-analytic manifolds}\label{subsec: LCRA}

We first give the definition of an LCRA manifold that we shall use:

\begin{definition}\label{def: LCRA manifold}
    \begin{enumerate}[(a)]
        \item A {\em locally conformally real-analytic} (LCRA) manifold {\em without boundary} is a Riemannian manifold $(M,g)$ without boundary such that about every point $p \in M$ there exists a coordinate chart $\phi : U \to \Omega \subseteq \R^n$ in which the metric $g$ takes the form
        \begin{equation}\label{LCRA metric form interior}
            g_{ij} = e^{2\alpha(x)} h_{ij}(x),
        \end{equation}
        where $\alpha \in C^{\8}(\Omega)$ and $h_{ij} \in C^{\omega}(\Omega)$ is a real-analytic metric on $\Omega$.

        \item A {\em locally conformally real-analytic manifold} {\em with boundary} is a Riemannian manifold $(M,g)$ with boundary whose interior is an LCRA manifold in the above sense, and moreover about any boundary point $p \in \d M$ there exists a boundary coordinate chart $\phi : U \to \overline{\Omega} \sub \R^n_{+}$ with $\phi(\d M \cap U) = \Gamma := \overline{\Omega} \cap \{ x_n = 0 \}$ in which the metric takes the form
        \begin{equation}\label{LCRA metric form boundary}
            g_{ij} = e^{2\alpha(x',x_n)} h_{ij}(x',x_n),
        \end{equation}
        where $\alpha \in C^{\8}(\overline{\Omega})$ and $h_{ij} \in C^{\omega}(\overline{\Omega})$ is real-analytic {\em up to the boundary} $\Gamma$. 
    \end{enumerate}
\end{definition}

\begin{remark}\label{remark: LCRA definition is more general}
    The above definition is slightly more general than the one given in \cite{LTS2022}, wherein it is required that the function $\a$ in \eqref{LCRA metric form boundary} additionally satisfies $\alpha(x',0) \in C^{\omega}(\Gamma)$. For the purposes of the present paper, the more general Definition \ref{def: LCRA manifold} shall suffice.
\end{remark}
\vspace{5pt}

By an LCRA coordinate chart, we mean any coordinate chart in which the metric $g$ takes the form \eqref{LCRA metric form interior} or \eqref{LCRA metric form boundary}; the latter we sometimes refer to as an LCRA boundary chart. Note that in Definition \ref{def: LCRA manifold}, we do not assume that $M$ has a real-analytic structure. However, it does indeed turn out that the interior of $M$ can be endowed with a real-analytic structure, which moreover contains LCRA coordinate charts in its atlas \cite[Proposition A.1]{LTS2022}:

\begin{proposition}\label{prop: LCRA manifolds have analytic structure}
    An LCRA manifold without boundary can be endowed with a real-analytic structure compatible with the original smooth structure. This real-analytic structure can be given by transforming from LCRA coordinates to harmonic coordinates. The metric continues to have the form \eqref{LCRA metric form interior} in the coordinates of the new atlas. 
\end{proposition}

We shall often assume that the interiors of our LCRA manifolds with boundary are endowed with the real-analytic structure given by Proposition \ref{prop: LCRA manifolds have analytic structure}. Doing so avoids the need to constantly work in coordinates when employing arguments that rely on analyticity.

We also have the following result, which says that if $M$ is an LCRA manifold with boundary such that $\alpha(x',0) \in C^{\omega}(\Gamma)$ in \eqref{LCRA metric form boundary} (that is, if $M$ is an LCRA manifold in the sense of \cite{LTS2022}), then the boundary $\d M$ inherits a real-analytic structure \cite[Proposition A.2]{LTS2022}: 

\begin{proposition} \label{prop: LCRA metric gives analytic structure on boundary}
    Suppose that $g$ is an LCRA metric such that in an LCRA boundary chart, the function $\alpha(x',x_n)$ given in \eqref{LCRA metric form boundary} additionally satisfies $\alpha(x',0) \in C^{\omega}(\Gamma)$. Then $\d M$ inherits a real-analytic atlas induced by harmonic coordinate charts for $g|_{\d M}$. The transition functions from coordinates on $\d M$ induced by LCRA coordinates to harmonic coordinates for $g|_{\d M}$ are real-analytic.  
\end{proposition}



We refer the reader to Appendix A in \cite{LTS2022} for proofs of the above propositions. For simplicity, we shall henceforth assume that the interiors of the LCRA manifolds that we consider are equipped with such a real-analytic structure. This will allow us to speak about analytic objects without the constant need to work in coordinates. Moreover, we have the following:

\begin{proposition}\label{prop: spinor bundle is analytic bundle}
    The spinor bundle of an LCRA manifold can be given the structure of a real-analytic bundle with respect to the real-analytic structure given by Proposition \ref{prop: LCRA manifolds have analytic structure}.
\end{proposition}

\begin{proof}
    Let $h$ be a real-analytic metric locally conformal to $g$. We can take as a local trivialization of $\bS_h$ the one induced by a local $h$-orthonormal frame, which is analytic on the base manifold. On the other hand, since $g$ is conformal to $h$, we have that $\bS_g$ is canonically isomorphic to $\bS_h$ (see the discussion at the end of Section \ref{subsec: spinors}). We thus obtain a system of local trivializations of $\bS_g$ whose transition functions are real-analytic on the base.
\end{proof}

When considering twisted Dirac operators, we shall assume that our auxiliary data $(E,\eta,A)$ are real-analytic in the interior of the manifolds considered. That is to say, $E$ has a system of local trivializations whose transition functions are real-analytic with respect to the $C^{\omega}$-structure on $M$ given by Proposition \ref{prop: LCRA manifolds have analytic structure}, and moreover the Hermitian metric $\eta$ and connection $1$-form $A$ are real-analytic in these trivializations. All that remains is to specify the behaviour of $(E,\eta,A)$ at the boundary, which motivates the following definition:

\vspace{5pt}
\begin{definition}\label{def: admissible}
    Let $(M,g)$ be an LCRA manifold with boundary, and let $(E,\eta)$ be a Hermitian bundle over $M$ endowed with a unitary connection $A$. We say that $(E,\eta,A)$ are {\em admissible data} on $(M,g)$ if $(E,\eta,A)$ are real-analytic in the interior of $M$ (with respect to the $C^{\omega}$-structure given by Proposition \ref{prop: LCRA manifolds have analytic structure}), and if the Hermitian metric $\eta$ and connection $1$-form $A$ are real-analytic up to the boundary with respect to a local trivialization of $E$ and an LCRA boundary chart.
\end{definition}

The requirements near the boundary are technical hypotheses on the auxiliary data $(E,\eta,A)$ that ensure our proof of Theorem \ref{main theorem} applies in the setting of twisted Dirac operators. Note that if $A$ is real-analytic up the boundary with respect to an LCRA boundary chart and local trivialization, one can find a real-analytic trivialization in which $A_n = 0$ near a point $p$ on the boundary. Indeed, this amounts to finding a smooth unitary isomorphism $\Phi : E \to E$ that solves
\begin{equation}
    \begin{cases}
        \d_n \Phi = [\Phi, A_n], \\
        \Phi|_{x_n = 0} = \id_E,
    \end{cases}
\end{equation}
in LCRA coordinates near $p$. In particular, since $A_n$ is analytic in these coordinates, so too is the solution $\Phi$ in a neighbourhood near $p$ by the Cauchy--Kovalevskaya theorem \cite[\S 16.4]{Taylor2011}. Applying the gauge transformation $\Phi$ then gives us a trivialization in which $A_n = 0$ near $p$. The orthonormal frame $(v_i)$ of $E$ corresponding to such a trivialization then satisfies $\nabla_{\d_n} v_i = 0$ near $p$. We call such a frame a {\em boundary normal frame} for $E$, and shall often work in such a frame near the boundary.

\vspace{10pt}
\section{Proof of the main result}\label{sec: proof full section}

In this section we prove Theorem \ref{main theorem}. First, in Section \ref{sec: extension}, we show that one can conformally rescale two metrics with equivalent BC maps in such a way that both metrics can be analytically extended past $\sU \sub \d M$ in boundary normal coordinates; further, the extended metrics are equal in the extended region. If our Dirac operators are twisted by admissible data, then the bundles and connections can be extended as well. We then show in Section \ref{sec: Greens functions determined} that the Green's kernels for the Dirac operators are equal in the extended region. Finally, in Section \ref{sec: reconstruction}, we show that we can embed both manifolds into a common Sobolev space using conformally rescaled Green's kernels. We finish the proof by showing that these embeddings have common image, and can be used to construct a conformal isomorphism between the two manifolds.


\vspace{10pt}
\subsection{Extending the manifolds in an analytic conformal gauge}\label{sec: extension}

As noted in Remark \ref{remark: principal symbol determines conformal class}, the principal symbol of the BC map determines the conformal class of the pullback metric on the boundary. In fact, much more is true: as was shown in \cite[Theorem A]{Valero2025}, in dimensions $n \geq 3$ the BC map determines the metric on the boundary up to a locally constant conformal factor. Moreover, if the mean curvature $H$ and all of its normal derivatives are given, then the BC map determines the Taylor series of the metric in boundary normal coordinates. Note that $H$ and its normal derivatives can be fixed by a conformal scaling. We quote here the following inductive form of Theorem A from \cite{Valero2025} that we shall use below:

\begin{theorem}\label{thm: boundary determination inductive form}
    Let $(M_1, g_1)$ and $(M_2, g_2)$ be compact spin manifolds with boundary of dimension $n \geq 3$ that are identified along $\sU \sub \d M_1,  \d M_2$. Suppose that $\Theta_{g_1, A_1, \sU} \sim \Theta_{g_2, A_2, \sU}$. Then $g_1|_{\sU}$ and $g_2|_{\sU}$ are related by a locally constant conformal factor. Moreover, let $H_i(x',x^n)$ denote the mean curvature of surfaces of constant $x^n$ is boundary normal coordinates for $g_i$ near $\sU$. Then, if
    \begin{equation}\label{derivatives of mean curvatures equal}
        \d_{n_1}^{k} H_1|_{\sU} = \d_{n_2}^{k} H_2|_{\sU}
    \end{equation}
    for all $0 \leq k \leq N$, it follows that
    \begin{equation}
        \d_{n_1}^{N+1} g^{\a \b}_1|_{\sU} = \d^{N+1}_{n_2} g^{\a \b}_2|_{\sU}.
    \end{equation}
    where $n_i$ denotes the normal coordinate with respect to $g_i$. Moreover, if \eqref{derivatives of mean curvatures equal} holds we also have
    \begin{equation}
        \d_{n_1}^{\l} A_1|_{\sU} = \d_{n_2}^{\l} A_2|_{\sU}
    \end{equation}
    for $0 \leq \l \leq N+1$, in trivializations of $E_1$ and $E_2$ where $(A_1)_n = (A_2)_n = 0$ near $\sU$.
\end{theorem}

Note that it does not follow from \cite[Theorem A]{Valero2025} that $\Theta_{g,A,\sU}$ determines $g|_{\sU}$ up to a locally constant conformal factor when the dimension $n = 2$. Moreover, when $n = 2$, it also does not follow from \cite[Theorem A]{Valero2025} the Taylor series of $A$ is determined by the BC map. This is the reason for the assumptions made in Theorem \ref{main theorem} in dimension $2$.

We now prove the following result, which shows that if two metrics have equivalent BC maps, then we can find suitable conformal rescalings of each so that they are real-analytic in boundary normal coordinates and have equal Taylor series on $\sU$. In particular, this will allow us to extend the rescaled metrics analytically to a larger manifold.

\begin{proposition}\label{prop: extension in analytic gauge}
    Let $(M_1, g_1)$ and $(M_2, g_2)$ be $n$-dimensional LCRA manifolds whose boundaries are identified over $\sU \sub \d M_1, \d M_2$, and whose BC maps are equivalent over $\sU$. If $n = 2$, then let us further assume that $g_1|_{\sU} = \lambda g_2|_{\sU}$ for some constant $\lambda$. Then there exist smooth positive functions $c_i$ on $M_i$ with $c_i|_{\sU}$ constant such that $c_i g_i$ is real-analytic up to the boundary in boundary normal coordinates and have equal Taylor series on $\sU$. In particular, there is a diffeomorphism $F : U_1 \to U_2$ near $\sU$ with $F|_{\sU} = \id$ such that $F^*(c_2 g_2) = c_1 g_1$.
\end{proposition}

\begin{proof}
    Let us first prove the proposition for $n \geq 3$. Since $g_1$ and $g_2$ are LCRA metrics, there exist coordinate charts $(U_i, \phi_i)$ for $M_i$ near $\sU \sub M_i$ so that $g_i$ is represented by $e^{2 \alpha_i} h_i$, where $h_i$ is a real-analytic metric. By Theorem \ref{thm: boundary determination inductive form}, we can rescale $g_1$ and $g_2$ by a constant so that $\alpha_1|_{\sU} = \a_2|_{\sU}$. This scaling preserves the equivalence of the BC maps. Now, shrinking $U_i$ if necessary, let $\beta_i \in C^{\8}(M_i)$ be such that $\beta_i|_{U_i} = \alpha_i$ and $\beta_i|_{\d M_i \setm \sU} = 0$ outside of $\sU$. In particular, we have $\b_1|_{\sU} = \b_2|_{\sU}$. Let $\tilde{g}_i := e^{-2\beta_i} g_i$. Since the BC maps of $g_1$ and $g_2$ are equivalent, we have
    \begin{align}
    \Theta_{\tilde{g}_1, A_1, \sU} &\sim e^{-\frac{n-1}{2} \beta_1|_{\sU}}  \Theta_{g_1, A_1, \sU} e^{\frac{n-1}{2} \beta_1|_{\sU}} \nonumber \\
    &\sim e^{-\frac{n-1}{2} \beta_2|_{\sU}} \Theta_{g_2, A_2, \sU} e^{\frac{n-1}{2} \beta_2|_{\sU}} \sim \Theta_{\tilde{g}_2, A_2, \sU}.
\end{align}
    where we have used equation \eqref{conformal change of BC map equivalence}. So the BC maps of $\tilde{g}_1$ and $\tilde{g}_2$ are also equivalent.

    Note that $\tilde{g}_i$ is real-analytic up to the boundary in coordinates $\phi_i$, as it is represented by $h_i$. Moreover, we have $\tilde{g}_1|_{\sU} = \tilde{g}_2|_{\sU}$. By Proposition \ref{prop: LCRA metric gives analytic structure on boundary}, we may choose coordinates $\phi^{\sU}$ on $\sU$ such that $\tilde{g}_1|_{\sU} = \tilde{g}_2|_{\sU}$ are real-analytic in coordinates $\phi^{\sU}$, and moreover the transition functions from $\phi^{\sU}$ to coordinates induced on $\sU$ by $\phi_i$ are real-analytic. Finally, we let $\tilde{\phi}_i$ be boundary normal coordinates for $\tilde{g}_i$ on $U_i$ induced by coordinates $\phi^{\sU}$ on $\sU$. It is easy to see that the transition functions $\tilde{\phi}_i \circ \phi_i^{-1}$ are real-analytic up to the boundary. We shall henceforth work in boundary normal coordinates $\tilde{\phi}_i : U_i \to \Omega$ for $\tilde{g}_i$, which are induced by common coordinates $\phi^{\sU}$ on $\sU$. Note that we assume without loss of generality that $\tilde{\phi}_1(U_1) = \tilde{\phi}_2(U_2) =  \Omega \subseteq \R^n_+$. Let $\Gamma := \Omega \cap \{ x_n = 0 \}$. The metrics $\tilde{g}_i$ are represented in coordinates $\tilde{\phi}_i$ by metrics $\tilde{h}_i \in C^{\omega}(\overline{\Omega})$.

Now, let $u_i \in C^{\8}(\Omega)$ be defined in coordinates $\tilde{\phi}_i$ as the solution to
\begin{equation}\label{BVP for conformal factor}
    \begin{cases}
        &2(n-1)\Delta_{\tilde{h}_i} u_i + (n-1)(n-2) |d u|^2_{\tilde{h}_i} -  S_{\tilde{h}_i} = 0, \\
        &u_i|_{\Gamma} = 0, \\
        &\d_{n_i} u_i |_{\Gamma} = H_{\tilde{h}_i}, 
    \end{cases}
\end{equation}
where $n_i$ is the normal to $\Gamma$ with respect to $\tilde{h}_i$ and $H_{\tilde{h}_i}$ is the mean curvature of $\Gamma$ with respect to $\tilde{h}_i$. Such a solution exists locally by the Cauchy--Kovalevskaya theorem since $\tilde{h}_i$ is real-analytic up to the boundary, and moreover the solution $u_i$ is also real-analytic up to the boundary \cite[\S 16.4]{Taylor2011}. Now let $v_i \in C^{\8}(M_i)$ be a smooth extension of $\tilde{\phi}_i^*u_i$ from $U_i$ to the rest of $M_i$, such that $v_i|_{\d M_i} = 0$ (where we shrink and relabel $U_i$ if necessary). Let $\bar{g}_i := e^{2 v_i} \tilde{g}_i$. We thus have by construction that
\begin{equation}
    \Theta_{\bar{g}_1, A_1, \sU} \sim \Theta_{\bar{g}_2, A_2, \sU}.
\end{equation}
Note that in coordinates $\tilde{\phi}_i$, the metric $\bar{g}_i$ is real-analytic up to the boundary $\Gamma$, since $u_i$ and $\tilde{h}_i$ are real-analytic up to the boundary in $\Omega$. Moreover, the scalar curvatures satisfy $S_{\bar{g}_i} = 0$ in $U_i$, and also the mean curvatures of $\sU$ with respect to $\bar{g}_i$ satisfy $H_{\bar{g}_i} = 0$. Indeed, the scalar and mean curvatures transform under conformal scaling according to \cite[1.J]{Besse1987}
\begin{equation}
    S_{\bar{g}_i} = e^{-2u_i} \left( S_{g_i} - 2(n-1) \Delta_{g_i} u_i - (n-1)(n-2) |du_i|^2_{g_i} \right),
\end{equation}
\begin{equation}
    H_{\bar{g}_i} = e^{-u_i}\left( H_{g_i} - \d_{n_i} u_i \right).
\end{equation}
Now, note that by the inductive form of the boundary determination result, Theorem \ref{thm: boundary determination inductive form}, since we have $\bar{g}_1|_{\sU} = \bar{g}_2|_{\sU}$ and $H_{\bar{g}_1}|_{\sU} = H_{\bar{g}_2}|_{\sU}$, it follows that  
\begin{equation}
    \d_{n_1} \bar{g}_1|_{\sU} = \d_{n_2} \bar{g}_2|_{\sU}
\end{equation}
in their respective boundary normal coordinates (here $n_i$ denotes the normal with respect to $\bar{g}_i$). It also follows from Theorem \ref{thm: boundary determination inductive form} that we have
\begin{equation}
    A_1|_{\sU} = A_2|_{\sU}, \ \ \ \ \ \d_{n_1} A_1|_{\sU} = \d_{n_2} A_2|_{\sU}, 
\end{equation}
with respect to boundary normal frames for $E_1$ and $E_2$.

Now note that in boundary normal coordinates for $g$, the scalar curvature $S_g$ takes the form
\begin{align}
    S_{g} &= g^{ij} \left( \d_{k} \Gamma^k_{\ ij} - \d_{j} \Gamma^k_{\ ik} + \Gamma^k_{\ ij} \Gamma^{\l}_{\ \l k} - \Gamma^{k}_{\ i \l} \Gamma^{\l}_{\ j k} \right) \nonumber \\
    &= -\d_n \Gamma^{\a}_{\ n \a} + g^{\a \b} \d_n \Gamma^n_{\ \a \b} + \cF_1\left( g, \d_n g \right) \nonumber \\
    &= 2 (n-1) \d_n H + \cF_1\left( g, \d_n g \right),
\end{align}
where $\cF_1$ is a scalar quantity depending tensorially on its arguments. Therefore, since $S_{\bar{g}_i} = 0$,
\begin{align}
    \d_{n_1} H_{\bar{g}_1}|_{\sU} &= -\frac{1}{2(n-1)} \cF_1\left( \bar{g}_1|_{\sU}, \d_{n_1} \bar{g}_1 |_{\sU} \right) \nonumber \\
    &= -\frac{1}{2(n-1)} \cF_1\left( \bar{g}_2|_{\sU}, \d_{n_2} {\bar{g}_2}|_{\sU} \right) = \d_{n_2} H_{\bar{g}_2}|_{\sU}.
\end{align}
From Theorem \ref{thm: boundary determination inductive form}, it then follows that 
\begin{equation}
    \d^2_{n_1} {\bar{g}_1}|_{\sU} = \d^2_{n_2} {\bar{g}_2}|_{\sU}, \ \ \ \ \ \d^2_{n_1} A_1|_{\sU} = \d^2_{n_2} A_2|_{\sU},
\end{equation}
in boundary normal coordinates. Now we proceed by induction. Suppose that
\begin{equation}
    \d^{k+1}_{n_1} {\bar{g}_1}|_{\sU} = \d^{k+1}_{n_2} {\bar{g}_2}|_{\sU}, \ \ \ \ \ \d^{k+1}_{n_1} A_1|_{\sU} = \d^{k+1}_{n_2} A_2|_{\sU},
\end{equation}
for all $k \leq N$. Differentiating the above expression for the scalar curvature, we find
\begin{align}
\d^N_{n_i} S_{\bar{g}_i} &= 2(n-1) \d^{N+1}_{n_i} H_{\bar{g}_i} + \cF_{N+1}\left( \bar{g}_i, \d_{n_i} {\bar{g}_i}, \dots, \d_{n_i}^{N+1} {\bar{g}_i} \right)
\end{align}
where $\cF_{N+1}$ is a scalar quantity depending tensorially on its arguments. Therefore, by the inductive hypothesis and the fact that $S_{\bar{g}_1} = S_{\bar{g}_2} = 0$ near $\sU$, we have
\begin{align}
    \d^{N+1}_{n_1} H_{\bar{g}_1}|_{\sU} &= -\frac{1}{2(n-1)} \cF_{N+1}\left( \bar{g}_i|_{\sU}, \d_{n_i} {\bar{g}_i}|_{\sU}, \dots, \d_{n_i}^{N+1} {\bar{g}_i}|_{\sU} \right) \nonumber \\
    &= -\frac{1}{2(n-1)} \cF_{N+1}\left( \bar{g}_i|_{\sU}, \d_{n_i} {\bar{g}_i}|_{\sU}, \dots, \d_{n_i}^{N+1} {\bar{g}_i}|_{\sU} \right) =\d^{N+1}_{n_2} H_{\bar{g}_2}|_{\sU}.
\end{align}
And so, by Theorem \ref{thm: boundary determination inductive form}, it follows that
$$\d_{n_1}^{N+2} {\bar{g}_1}|_{\sU} = \d_{n_2}^{N+2}{\bar{g}_2}|_{\sU}, \ \ \ \ \ \d_{n_1}^{N+2} A_1|_{\sU} = \d_{n_2}^{N+2} A_2|_{\sU}.$$
As this holds for all $N$, it follows that the Taylor series of the rescaled metrics $\bar{g}_1$ and $\bar{g}_2$ in their respective boundary normal coordinates are equal, as are the Taylor series of $A_1$ and $A_2$. Proposition \ref{prop: extension in analytic gauge} therefore follows with the diffeomorphism $F$ mapping one set of boundary normal coordinates to the other. As both are induced by the same coordinates on $\sU$, we have $F|_{\sU} = \id$. 

Let us now consider the case $n = 2$. In this case, by assumption we have that $g_1|_{\sU} = \lambda g_2|_{\sU}$ for some constant $\lambda$. We can therefore rescale as above to obtain $g_1|_{\sU} = g_2|_{\sU}$ while preserving the equivalence of BC maps. The rest of the proof is as above, except that the BC maps play no role in the boundary determination. That is, we solve the boundary value problem \eqref{BVP for conformal factor} (which amounts here to locally prescribing the Gaussian curvature and geodesic curvature of the boundary) just as in the case $n \geq 3$, from which it follows inductively that the rescaled metrics have equal Taylor series on the boundary and equivalent BC maps (alternatively, one could also appeal to the fact that after conformal rescaling, both metrics are isometric to the Euclidean metric on $\R^2$).
\end{proof}

\begin{remark}\label{remark: linear equivalence extension result}
    Note that if we were to require linear equivalence of BC maps in dimension $2$, then we would not need the assumption that $g_1|_{\sU} = \lambda g_2|_{\sU}$ for constant $\lambda$, as we could always conformally rescale our metrics to achieve this, and the notion of linear equivalence is preserved by arbitrary conformal rescalings regardless of their behaviour on the boundary (see Remark \ref{remark: linear equivalence}). On the other hand, little would be lost in higher dimensions, since Proposition \ref{prop: linear vs canonical equivalence} implies that if two metrics have linearly equivalent BC maps, then we can always conformally rescale our metrics on the boundary so that the BC maps become (canonically) equivalent. 
\end{remark}
\vspace{5pt}

With the preceding proposition, we can extend our LCRA manifolds as follows. First, let us conformally rescale our metrics by positive functions $c_i$ as in Proposition \ref{prop: extension in analytic gauge}, so that the new metrics $\bar{g}_i$ are real-analytic in boundary normal coordinates and have equal Taylor series on $\sU$. In particular, after identifying $M_1$ and $M_2$ near $\sU$ via the diffeomorphism $F$ from Proposition \ref{prop: extension in analytic gauge}, we have that $\bar{g}_1 = \bar{g}_2$ near $\sU$. We extend the manifold $M_i$ to a larger manifold $\tilde{M}_i$ by attaching a small ball $B_{\varepsilon} \sub \R^n$ of radius $\varepsilon$ in boundary normal coordinates of $\bar{g}_i$ so that points of $B_{\varepsilon} \cap \R^n_+$ are identified with points in $M_i$. Let $W_i = \tilde{M}_i^{\mathrm{int}} \setm M_i$. By choosing $\varepsilon$ small enough we can assume that $W_1$ and $W_2$ are identified via a diffeomorphism $j : W_1 \to W_2$, which extends the diffeomorphism $F$ from Proposition \ref{prop: extension in analytic gauge}. We can therefore identify $W_1$ and $W_2$ via $j$, and denote the extended region by $W$. We will usually identify $W_1$ and $W_2$ in this way, though sometimes in the interest of clarity we may explicitly recall the diffeomorphism $j : W_1 \to W_2$. 

Shrinking $W$ if necessary, we also have that the metrics $\bar{g}_i$ extend to metrics $\tilde{g}_i$ on $\tilde{M}_i$ such that $\tilde{g}_1|_W = \tilde{g}_2|_{W}$. Moreover, over $W$ we have the identification $j_S : \tilde{\bS}_1|_W \to \tilde{\bS}_2|_W$, which extends the identification of spinor bundles over $\sU$ given by the equivalence of BC maps. The auxiliary bundles $E_i$ also extend to bundles $\tilde{E}_i$ trivially over $W$, such that by identifying boundary normal frames we have an isomorphism $j_E : \tilde{E}_1|_{W} \to \tilde{E}_2|_{W}$ that extends the isomorphism $\chi : E_1|_{\sU} \to E_2|_{\sU}$. Since the connections $A_1, A_2$ are real-analytic in such frames and have equal Taylor series, it follows that they too extend to connections $\tilde{A}_1, \tilde{A}_2$ such that $\tilde{A}_1 = j_E^* \tilde{A}_2$. Note that by construction all of these isomorphisms hold not only over the extended region $W$, but also over a neighbourhood $U$ of $\sU$ in the original manifolds $M_i$, as given by Proposition \ref{prop: extension in analytic gauge}.

We shall use the notation $\cE_i := \bS_i \otimes E_i$ and $\tilde{\cE}_i := \tilde{\bS}_i \otimes \tilde{E}_i$ for the twisted spinor bundles over $M_i$ and $\tilde{M}_i$, respectively. Since $\tilde{\cE}_1|_{W}$ and $\tilde{\cE}_2|_W$ are identified via $j_{\cE} := j_S \otimes j_E$ as described above, we will often denote both simply by $\tilde{\cE}|_W$. We also remark that the $\tilde{M}_i$ is LCRA, and so by Proposition \ref{prop: LCRA manifolds have analytic structure} its interior has a real-analytic structure, which extends the real-analytic structure on the interior of $M_i$. In particular, it makes sense to consider $\tilde{E}_i$, $\tilde{\bS}_i$, and $\tilde{\cE}_i$ as real-analytic bundles over the interior of $\tilde{M}_i$, and speak of real-analytic sections of these bundles.

\vspace{5pt}

\begin{remark}
    It is worth taking a moment to compare the above approach to that found in \cite{LTS2022} for the conformal Laplacian. In \cite{LTS2022}, the metrics are first conformally rescaled so that the quantities in \eqref{derivatives of mean curvatures equal} are null. There is, however, no guarantee that the metrics so obtained will be represented by analytic metrics in boundary normal coordinates; indeed, in general they will not be. Thus, in order to determine the metrics near the boundary, the authors employ a clever approach in which they use their (slightly stronger) hypothesis of local conformal real-analyticity to define what they dub $Z$-coordinates, and work within these to show that the DN map of the conformal Laplacian determines the metrics near the boundary, up to a suitable conformal factor. The approach we have taken above for the Dirac operator, on the other hand, appears to be simpler, more straightforward, and more closely resembles the approach based on extending the manifold that has found much success in other Calder{\'o}n-type problems for analytic metrics \cite{LTU2003, KLU2011, EV2024, Gabdurakhmanov2025, Valero2025}. One reason for this is our judicious use of the conformal symmetry present in our problem: a key idea in the proof of Proposition \ref{prop: extension in analytic gauge} given above is that it suffices to have equality in \eqref{derivatives of mean curvatures equal}, rather than the stronger condition that both sides be $0$. We also point out that since the Dirac operator is first-order, equivalence of BC maps is sensitive only to the boundary values of the conformal factor, but not its higher derivatives. In contrast, equality of DN maps for the conformal Laplacian is preserved by conformal scalings by $u \in C^{\8}(M)$ only if both $u|_{\d M}$ and $\d_n u|_{\d M}$ vanish. Hence, the method of using the second-order boundary value problem \eqref{BVP for conformal factor} to find a suitable conformal scaling that leads to analytic metrics with equal Taylor series is not directly available for the conformal Laplacian, since being able to freely specify the normal derivative of the conformal factor is a key part of the argument. However, since the argument is entirely local, and since the DN map of the conformal Laplacian determines the mean curvature \cite{LTS2022}, a higher-order version of \eqref{BVP for conformal factor} obtained by taking normal derivatives is likely sufficient to proceed with the extension argument given above, and so treat the Calder{\'o}n problem for the conformal Laplacian, as well as other conformally covariant operators, on LCRA manifolds with partial data.


\end{remark}


\vspace{10pt}
\subsection{Conformally rescaled Green's kernels}\label{sec: Greens functions determined}

Having shown in Section \ref{sec: extension} that we can extend suitable conformal representatives of $(M_i, g_i)$ to $(\tilde{M}_i, \tilde{g}_i)$ so that $\tilde{g}_1$ and $\tilde{g}_2$ agree in $\tilde{M}_i \setm M_i$, we now want to consider the MIT Green's kernels for the Dirac operators on the extended manifolds $\tilde{M}_i$ and show that they agree in the extended region. We will then define a rescaled version of these Green's kernels in such a way to obtain a good analytic object that we will need to finish the Proof of Theorem \ref{main theorem} in Section \ref{sec: reconstruction}. 


We first introduce the Green's kernel for the Dirac operator with MIT boundary conditions. Let $\pi_L : M \times M \to M$ and $\pi_R : M \times M \to M$ be the left and right projections, respectively. For vector bundles $\cV$ and $\cW$ over $M$, we define the {\em outer tensor product} $\cV \boxtimes \cW := \pi_L^* \cV \otimes \pi_R^* \cW$ over $M \times M$. In particular, the fiber of $\cE \boxtimes \cE^*$ over $(x,y) \in M \times M$ is equal to $\Hom(\cE_y, \cE_x)$, the space of linear maps from $\cE_y$ to $\cE_y$. 

\begin{definition}
    Let $(M,g)$ be a compact spin manifold with boundary $\d M$, and let $D_A$ be a twisted Dirac operator on $\cE := \bS \otimes E$ over $M$. The {\em Green's kernel} for $D_A$ with MIT$^+$ boundary conditions is the distributional section $G$ of $\cE \boxtimes \cE^*$ over $M \times M$ satisfying
    \begin{equation}\label{Greens kernel BVP definition}
    \begin{cases}
        D_A G(x,y) = \delta_y \id_{\cE_y}, \\
        \bP^+ G(x,y)|_{\d M} = 0,
    \end{cases}
\end{equation}
for fixed $y \in M^{\mathrm{int}}$, where the operators $D_A$ and $\bP^+$ act on the first factor $\cE_x$. One can define a Green's kernel $G^{-}(x,y)$ for MIT$^-$ boundary conditions by replacing $\bP^{+}$ with $\bP^{-}$ in \eqref{Greens kernel BVP definition}.
\end{definition}

Note that by elliptic regularity, $G(x,y)$ is a smooth section of $\cE \boxtimes \cE^*$ away from the diagonal. Moreover, using the Green's identity for the Dirac operator \eqref{Green's identity for D_A}, one can show that the definition of $G(x,y)$ given above is equivalent to requiring that
\begin{equation}
    \int_M \< G(x,y) \psi_1(y), D_A \psi_2(x) \>\, dx = \< \psi_1(y), \psi_2(y) \>
\end{equation}
for all smooth sections $\psi_1, \psi_2$ of $\cE$ such that $\bP^+ \psi_i = 0$ (c.f. \cite{Raulot2011}).

One can show that the conformal covariance of the Dirac operator \eqref{Dirac operator under conformal change exp version} gives rise to a similar conformal covariance in the Green's kernel. Indeed, we have the following \cite{Raulot2011}:
\begin{proposition}\label{prop: Greens kernel under conformal scaling}
    Let $\bar{g} = e^{2u} g$. Then the Dirac Green's kernels are related by
    \begin{equation}\label{Greens kernels under conformal scaling}
    G_{\bar{g}}(x,y) = e^{-\frac{n-1}{2} u(x)} e^{-\frac{n-1}{2} u(y)} \bar{G}_{g}(x,y)
    \end{equation}
    where the bar on the right-hand-side indicates the canonical isomorphism $\bS_{g} \to {\bS}_{\bar{g}}$.
\end{proposition}

We direct the reader to \cite[Proposition 3]{Raulot2011} for a proof of the above proposition. We also have the following expected asymptotic behaviour for the Dirac Green's kernels as $x \to y$, which follows from the study of fundamental solutions for Dirac systems carried out in \cite{DinhSon2012}:

\begin{proposition}\label{prop: Greens asymptotics}
    There exists $\delta > 0$ and constants $c_n, c_n' > 0$ such that as $x \to y$, we have
\begin{equation}\label{asymptotics of Greens kernel}
    G(x,y) =  \frac{c_n \gamma(e_r)}{d(x,y)^{n-1}} + O\left( d(x,y)^{1-n + \delta} \right),
\end{equation}
\begin{equation}\label{asymptotics of Greens kernel derivatives}
    \nabla_x G(x,y) =  \frac{c'_n \gamma(e_r)}{d(x,y)^{n}} + O\left( d(x,y)^{-n + \delta} \right),
\end{equation}
where $d(x,y)$ is the geodesic distance, and $e_r$ is the radial unit vector near $y$.
\end{proposition}

Finally, we note that because the Dirac operator with MIT boundary conditions is not self-adjoint, we do not have a notion of symmetry for the Green's kernels under the interchange of $x$ and $y$, unlike the Green's kernel for chiral boundary conditions \cite{Raulot2011, Valero2025}. Rather, the adjoint of $D_A$ with MIT$^\pm$ conditions is $D_A$ with MIT$^{\mp}$ conditions. We thus have \cite{Raulot2011}:

\begin{proposition}\label{prop: symmetry of  MIT Greens kernel}
    Let $G^{-}(x,y)$ be the Green's kernel for MIT$^-$ conditions. Then we have
    \begin{equation}\label{symmetry of MIT Greens kernel}
    G(x,y)^* = G^{-}(y,x)
    \end{equation}
    as elements in $\Hom(\cE_x, \cE_y)$, whenever $x \neq y$.
\end{proposition}

Let us now return to the proof of Theorem \ref{main theorem}. In Section \ref{sec: extension}, we showed that we can extend our manifolds, metrics, bundles, and connections in an analytic fashion after a suitable conformal rescaling. Let $\tilde{G}_i(x,y)$ be the Green's kernel for the {\em extended} manifold $(\tilde{M}, \tilde{g}_i)$ and extended auxiliary data. We now want to show that the equivalence of BC maps on $\sU$ implies that $\tilde{G}_1(x,y)$ and $\tilde{G}_2(x,y)$ are in fact equal in the extended region $W$, once we make the usual identifications of our bundles over $W$. The following proposition is essentially the same as \cite[Proposition 4.3]{Valero2025}:

\begin{proposition}\label{prop: Greens functions agree in U}
    We have $\tilde{G}_1(x,y) = \tilde{G}_2(x,y)$ in $\Hom(\tilde{\cE}_y, \tilde{\cE}_x)$ for all $x,y \in W$ with $x \neq y$.
\end{proposition}

\begin{proof}
    For $y \in W$ and $\psi \in \tilde{\cE}_y$, we define $\Xi_{y, \psi}$ to be the unique solution to the following boundary value problem on the unextended manifold $M_2$,
    \begin{equation}
        \begin{cases}
            D_{g_2, A_2} \Xi_{y, \psi} = 0 \\ 
            \bP^+ \Xi_{y, \psi}|_{\sU} = \bP^+ \tilde{G}_1(x,y)\psi, \\
            \bP^+ \Xi_{y, \psi}|_{\d M_2 \setm \sU} = 0.
        \end{cases}
    \end{equation}
On the other hand, since $y \in W$, we have that $\tilde{G}_1(x,y) \psi$ satisfies
\begin{equation}
        \begin{cases}
            D_{g_1, A_1} \tilde{G}_1(x,y) \psi = 0 \\ 
            \bP^+ \tilde{G}_1(x,y) \psi |_{\sU} = \bP^+ \tilde{G}_1(x,y) \psi, \\
            \bP^+ \tilde{G}_1(x,y) \psi|_{\d M_1 \setm \sU} = 0,
        \end{cases}
    \end{equation}
on $M_1$. Therefore, as the boundary conjugation maps $\Theta_{g_1, A_1, \sU}$ and $\Theta_{g_2, A_2, \sU}$ are equivalent, upon making the usual identifications of our extended bundles near $\sU$, we have
\begin{equation}
    \Theta_{g_2, A_2, \sU}\big(\, \bP^+ \Xi_{y, \psi}\big) = \Theta_{g_1, A_1, \sU}\big(\, \bP^{+} \tilde{G}_1(x,y) \psi \big) = \bP^{-} \tilde{G}_1(x,y) \psi.
\end{equation}
Therefore, $\Xi_{y, \psi}|_{\sU} = \tilde{G}_1(x,y)\psi|_{\sU}$. So we can define the following section of $\tilde{\cE}_2$ over $\tilde{M}_2$,
\begin{equation}
    \tilde{\Xi}_{y,\psi} = \begin{cases}
        \tilde{G}_1(x,y)\psi, & x \in \tilde{M}_2 \setm M_2, \\
        \Xi_{y,\psi}(x), & x \in M_2.
    \end{cases}
\end{equation}
Note that by construction, $\tilde{\Xi}_{y,\psi}$ is a distributional solution of 
\begin{equation}
    \begin{cases}
        D_{\tilde{g}_2, \tilde{A}_2} \tilde{\Xi}_{y, \psi} = \psi\, \delta_y, \\
        \bP^+  \tilde{\Xi}_{y,\psi} |_{\d \tilde{M}_2} = 0.
    \end{cases}
\end{equation}
Therefore, by uniqueness, we must have $\tilde{\Xi}_{y, \psi} = \tilde{G}_2(x,y) \psi$. Since this is true for all $\psi \in \tilde{\cE}_y$, we have $\tilde{G}_1(x,y) = \tilde{G}_2(x,y)$ for $x, y \in W$ with $x \neq y$. This completes the proof.
\end{proof}


Motivated by \cite{LTS2022}, we now want to define a rescaled version of these Green's kernels to give us something analytic in the interior of our manifolds. To this end, let $\Gamma \sub \d \tilde{M}_i \setm \d M_i$ be an open subset of the common extended boundary, and let $\rho : \Gamma \to \R$ be a smooth non-negative function with compact support in $\Gamma$, and which is positive on some open subset of $\Gamma$. We then define a scalar function $P_i(x)$ on the interior of $\tilde{M}_i$ by 
\begin{equation}\label{P_i definition}
    P_i(x) := \left( \int_{\Gamma} \rho(z)\, \Tr_{\tilde{\cE}_z} \left( \tilde{G}_i(x,z)^* \tilde{G}_i(x,z) \right)\, dz  \right)^{\frac{1}{2}}.
\end{equation}
Note that $P_i : \tilde{M}_i^{\mathrm{int}} \to \R$ is well-defined and smooth. Indeed, for $x \in \tilde{M}_i^{\mathrm{int}}$, the integrand in \eqref{P_i definition} is a smooth and positive function of $x$. Moreover, we claim that $P_i(x) > 0$ for all $x \in \tilde{M}_i^{\mathrm{int}}$. Indeed, if $P_i(x) = 0$ for some $x \in \tilde{M}_i^{\mathrm{int}}$, then it follows from \eqref{P_i definition} that we must have $\tilde{G}_i(x,z) = 0$ for all $z$ in some open subset of $\Gamma$. But by \eqref{symmetry of MIT Greens kernel}, $\tilde{G}_i(x,z)^*$ satisfies a Dirac equation with respect to $z$, and so by the unique continuation principle \cite{BW1993}, it must be that $\tilde{G}_i(x,z)$ is identically $0$ for $x \neq z$. This however contradicts the asymptotics \eqref{asymptotics of Greens kernel}. We therefore have $P_i(x) > 0$ for al $x \in \tilde{M}_i^{\mathrm{int}}$. Note that this argument continues to hold for $x \in \d \tilde{M}_i \setm \Gamma$.

We can therefore define the {\em rescaled Green's kernel},
\begin{equation}\label{rescaled Green's kernel first appearance}
    K_i(x,y) := \frac{\tilde{G}_i(x,y)}{P_i(x)P_i(y)},
\end{equation}
which like $\tilde{G}_i(x,y)$ is a smooth section of $\tilde{\cE} \boxtimes \tilde{\cE}^*$ for $x,y \in \tilde{M}_i^{\mathrm{int}}$ and $x \neq y$. Moreover, since the Green's kernels are equal near $\Gamma$ by Proposition \ref{prop: Greens functions agree in U}, it follows that $P_1 = P_2$ in $W$, and so we also have $K_1(x,y) = K_2(x,y)$ for $x,y \in W$ and $x \neq y$. Note that if we conformally rescale our metric $\bar{g}_i = e^{2u} \tilde{g}_i$ such that $u$ equals $0$ near $\Gamma$, then by \eqref{Greens kernels under conformal scaling} we have
\begin{equation}
   P_{\bar{g}_i}(x) = e^{-\frac{n-1}{2} u(x)} P_{g_i}(x),
\end{equation}
where we have written the dependence on the metric explicitly in the subscripts for clarity. Thus, under conformal scaling by such a conformal factor, by \eqref{Greens kernels under conformal scaling} the rescaled Green's kernels satisfy
\begin{equation}\label{K is invariant under conformal}
    K_{\bar{g}_i}(x,y) =  K_{g_i}(x,y),
\end{equation}
again writing the dependence of the rescaled Kernels on the metric explicitly.

We have the following important lemma:
\begin{lemma}\label{lemma: K_i is analytic}
    $K_i(x,y)$ is a real-analytic section of $\tilde{\cE} \boxtimes \tilde{\cE}^*$ for $x,y \in \tilde{M}_i^{\mathrm{int}}$ and $x \neq y$.
\end{lemma}
\begin{proof}
    By definition of the $C^{\omega}$-structure on the interior of $\tilde{M}_i$, it suffices to verify that $K_i(x,y)$ is represented by an analytic matrix with respect to an LCRA coordinate chart, and real-analytic trivializations of $\tilde{E}_i$ and $\tilde{\bS}_i$, about a point $x_0$, when $x \neq y$. In such a local representation, the metric $\tilde{g}_i$ is represented by $e^{2 \alpha_i} h_i$, where $h_i$ is real-analytic. However, by equation \eqref{K is invariant under conformal} we have (upon making the canonical identification of our spinor bundles, as usual)
    \begin{equation}\label{K_i is analytic eqn}
        K_{\tilde{g}_i}(x,y) = K_{g_i'}(x,y),
    \end{equation}
    where $g_i' := e^{-2u_i} g_i$ with $u_i$ defined to be $0$ near $\Gamma$, and equal to $\alpha_i$ near $x$. By definition, however, the rescaled Green's kernel $K_{g_i'}(x,y)$ takes the form 
    \begin{equation}
        K_{g_i'}(x,y) = \frac{G_{g_i'}(x,y)}{P_{g'_i}(x) P_{g'_i}(y)}.
    \end{equation}
    Note that the Green's kernel $G_{g_i'}(x,y)$ satisfies the equation
    \begin{equation}
        D_{g_i', A_i} G_{g_i'}(x,y) = 0,
    \end{equation}
    in $x$, which in a local presentation as above is a linear elliptic system with real-analytic coefficients. Therefore, it follows from elliptic regularity that $G_{g_i'}(x,y)$ is real-analytic in $x$. From \eqref{symmetry of MIT Greens kernel}, a similar argument shows that $G_{g_i'}(x,y)$ is also real-analytic in $y$. Now we claim that the function $P_{g_i'}(x)$ given by \eqref{P_i definition} is real-analytic in $x$. Indeed, let us consider the restriction of $P_{g_i'}$ to a suitably small neighbourhood $V$ of $x_0$. The integrand in \eqref{P_i definition} is
    \begin{equation}
        f(x,z) := \rho(z)\, \Tr_{\tilde{\cE}_z} \left( \tilde{G}_{g_i'}(x,z)^* \tilde{G}_{g_i'}(x,z) \right),
    \end{equation}
    considered as a function on $V \times \overline{\Gamma}$ where $\Gamma \sub \d \tilde{M}_i \setm M_i$ and $\rho \in C^{\8}(\Gamma)$ are given in the definition of $P_i$ above. Since $G_{g_i'}(x,y)$ is real-analytic in $x$, it follows that $f(x,z)$ is analytic in $x$ for each fixed $z \in \bar{\Gamma}$. Choosing coordinates centered at $x_0$, it therefore suffices to show that
    \begin{equation}
        F(x) := \int_{\bar{\Gamma}} f(x,z)\, dz
    \end{equation}
    is a real-analytic function in a neighbourhood of $0 \in \R^n$. This follows from the theory of analytic functions, for which we refer the reader to \cite{Krantz2002primer, range1998holomorphic}, for example. Here we only sketch the proof. Since $f(x,z)$ is real-analytic for each fixed $z \in \bar{\Gamma}$, it follows that $f(\cdot,z)$ has a holomorphic extension $f_{\C}(\cdot, z)$ to a neighbourhood of $0 \in \C^n$ for each $z \in \bar{\Gamma}$. Since $z$ varies over a compact set, we can take $f_{\C}(\cdot, z)$ to be a family of functions on a fixed neighbourhood $\cN$ of $0$ in $\C^n$. If we define
    \begin{equation}
        F_{\C}(x) := \int_{\bar{\Gamma}} f_{\C}(x,z)\, dz
    \end{equation}
    for $x \in \cN$, then $F$ is the restriction of $F_{\C}$ to $\R^n$. It thus suffices to show that $F_{\C}$ is holomorphic on $\cN$. By Hartog's theorem, it suffices to show that it is separately holomorphic in each complex variable. By Morera's theorem, it suffices to show that 
    \begin{equation}
        \int_{\gamma} F_{\C}(x_1, \dots, x_n)\, dx_1 = 0
    \end{equation}
    for all closed curves $\gamma \sub \C$ and fixed $x_2, \dots, x_n$ in $\cN$. By Fubini's theorem and the holomorphicity of $f_{\C}(\cdot, z)$ for each $z \in \bar{\Gamma}$, we may deduce that
    \begin{equation}
        \int_{\gamma} F_{\C}(x)\, dx_1 = \int_{\gamma} \int_{\bar{\Gamma}} f_{\C}(x,z)\, dz\, dx_1 = \int_{\bar{\Gamma}} \int_{\gamma} f_{\C}(x,z)\, dx_1\, dz = 0
    \end{equation}
    so long as $|f_{\C}(x,z)|$ is jointly integrable over $\cN \times \bar{\Gamma}$. Since this is clearly true of of $|f(x,z)|$, by shrinking $\cN$ if necessary, it is also true of $f_{\C}(x,z)$. Thus, $F_{\C}$ is holomorphic, and hence we conclude that $F$ is real-analytic in a neighbourhood of $0$. Therefore, $P_{g_i'}(x)$ is indeed real-analytic near $x_0$. Since one can make the same argument for $y$, it follows that the rescaled Green's kernel $K_{g_i'}(x,y)$ is real-analytic when $x \neq y$, and so the lemma is proved.
\end{proof}

Finally, we have the following result regarding the asymptotics of the rescaled Green's kernels, which follows immediately from Proposition \ref{prop: Greens asymptotics}:

\begin{proposition}\label{prop: rescaled asymptotics}
    For $x \in \tilde{M}_i^{\mathrm{int}}$, we have
    \begin{equation}\label{rescaled asymptotics}
        \lim_{x \to y} \frac{K_i(x,y)}{d_i(x,y)^{1-n}} = \frac{c_n \gamma(e_r)}{P_i(x)^2},
    \end{equation}
    where $d_i(x,y)$ is the distance between $x$ and $y$ with respect to $\tilde{g}_i$, $c_n$ is the constant from Proposition \ref{prop: Greens asymptotics}, and $e_r$ is the radial unit vector in geodesic normal coordinates about $x$.
\end{proposition}

\vspace{5pt}
\subsection{Constructing a conformal map from embeddings into a Sobolev space}\label{sec: reconstruction}

In this section, we want to use the rescaled Green's functions to construct an embedding of our bundles into a common Sobolev space. These embeddings can then be used to construct a conformal equivalence between the interiors of the manifolds. It will then follow from our choice of functions $P_i$ that this equivalence extends smoothly to the boundaries of the original manifolds. To this end, for $i = 1,2$ and $\varepsilon > 0$, we define the following sets,
\begin{equation}
    \widehat{M}_i := M_i \cup \{ x \in \tilde{M}_i  \ | \ d_i(x, \d \tilde{M}_i) \geq \varepsilon \},
\end{equation}
and choose $\varepsilon$ small enough so that $\widehat{M}_i$ is connected and has nonempty intersection with the open set $W$ defined at the end of Section \ref{sec: extension}. Note that $\widehat{M}_i$ so defined has as its boundary $\d M_i \setm \sU$ plus a part lying entirely in $W$. We also use $\widehat{\bS}_i$ to denote the restriction of $\tilde{\bS}_i$ to $\widehat{M}_i$, and similarly for $\widehat{E}_i$ and $\widehat{\cE}_i$. For $s < -\frac{n}{2}$ we then define the map
\begin{align}\label{embedding definition}
\begin{split}
    \cK_i : \widehat{\cE}_i &\to H^s(W, \tilde{\cE}_i|_W) \\
    \psi_y &\mapsto K_i(\cdot,y)\psi_y
\end{split}
\end{align}
where $\psi_y$ is an element in the fiber of $\widehat{\cE}_i$ over $y \in \widehat{M}_i$. Such a map is well-defined by our choice of $s$, and since the functions $P_i$ are smooth up to the boundary away from $\Gamma$. Moreover, $\cK_i$ a $C^1$ map into $H^s(W, \tilde{\cE}_i|_W)$ (see \cite{Valero2025} for example). In the following, we let $\widehat{\cE}_i^0$ denote the restriction of $\widehat{\cE}_i$ to $\widehat{M}_i^{\mathrm{int}}$ with the zero section removed. We then have the following:

\begin{lemma}
    For each $y \in \widehat{M}^{\mathrm{int}}_i$, the restriction of $\cK_i : \widehat{\cE}_i \to H^s(W, \tilde{\cE}_i|_W)$ to the fiber $(\widehat{\cE}_i)_y$ is a linear embedding. Moreover, $\cK_i$ is an injective $C^1$ immersion of $\widehat{\cE}^0_i$ into $H^s(W, \tilde{\cE}_i|_W)$.
\end{lemma}

\begin{proof}
    First we show that $\cK_i$ is a linear embedding on each fiber. Indeed, suppose that $\cK_i(\psi_y) = 0$. Then we must have $K_i(x,y)\psi_y = 0$ for all $x \in (W \cap \widehat{M}_i) \setm \{y\}$. By analytic continuation, this must hold for all $x \in \widehat{M}_i \setm \{y\}$. But this contradicts the asymptotics \eqref{rescaled asymptotics}. Hence, we have that the restriction of $\cK_i$ to each fiber is a linear embedding.

    We now prove that the restriction of $\cK_i$ to $\widehat{\cE}_i^0$ is an injective immersion. The proof that the derivative $d \cK_i$ is injective at every point in $\widehat{\cE}_i^{0}$ is completely analogous to the corresponding result in \cite{Valero2025} and also follows from the asymptotics and real-analyticity of the rescaled Green's kernels. Here we shall only prove the injectivity of $\cK_i$ to give an idea of the proof. Thus, suppose that $\cK_i(\psi) = \cK_i(\psi')$ for $y,y' \in \widehat{M}_i$ and non-zero $\psi, \psi'$ in the fiber of $\widehat{\cE}_i$ over $y,y'$, respectively. Then we have $K_i(x,y)\psi = K_i(x,y')\psi'$ for all $x \in (W \cap \widehat{M}_i) \setm \{y,y'\}$. By real-analyticity of the rescaled Green's kernels, we must thus have $K_i(x,y)\psi = K_i(x,y')\psi'$ for all $x \in \widehat{M}_i \setm \{y,y'\}$. This contradicts the asymptotics \eqref{rescaled asymptotics} unless $y = y'$. Since $K_i(\cdot, y)$ is a linear embedding on $(\widehat{\cE}_i)_y$, it follows that we have $\psi = \psi'$. Therefore $\cK_i$ is injective, and the lemma follows.
\end{proof}

As mentioned after equation \eqref{rescaled Green's kernel first appearance}, Proposition \ref{prop: Greens functions agree in U} and the construction of the functions $P_i$ imply that $K_1(x,y) = K_2(x,y)$ for $x,y \in W$ and $x \neq y$, upon making the usual identification of our bundles. In particular, it follows that $\cK_1|_W = \cK_2|_W$ as maps into $H^s(W, \tilde{\cE}_i|_W)$. More precisely, writing out the isomorphism $j_{\cE} : \widehat{\cE}_i|_W \to \widehat{\cE}_i|_W$ explicitly, we have
\begin{equation}\label{intertwining relation on W}
    \cK_2|_W = j_{\cE} \circ \cK_1|_W \circ j_{\cE}^{-1}.
\end{equation}

Given the above, we now want to prove the following theorem, which is similar to corresponding theorems in \cite{LTU2003, KLU2011, LTS2022, EV2024, Gabdurakhmanov2025, Valero2025}, and from which Theorem \ref{main theorem} shall follow. In fact, the proof that we give below is essentially an adaptation of the proof given in \cite{Valero2025} to the LCRA setting.

\begin{theorem}\label{thm: reconstruction step}
    Let $\cK_i : \widehat{\cE}_i \to H^s(W, \tilde{\cE}_i|_W)$ be as above. Then we have $\cK_2(\widehat{\cE}_2) = j_{\cE} (\cK_1(\widehat{\cE}_1))$, and the map $\cK_2^{-1} \circ j_{\cE} \circ \cK_1 : \widehat{\cE}^0_1 \to \widehat{\cE}_2^0$ extends to a unitary isomorphism $\Phi : \widehat{\cE}_1 \to \widehat{\cE}_2$ that is real-analytic in $\widehat{M}_i^{\mathrm{int}}$, and which covers a conformal equivalence $\vp : (\widehat{M}_1, \tilde{g}_1) \to (\widehat{M}_2, \tilde{g}_2)$, such that $\Phi|_W = j_{\cE}$. In addition, there exists a complex line bundle $L$ on $\widehat{M}_2$, a flat connection $\nabla^L$ on $L$, and a unitary isomorphism $\Phi_E : \widehat{E}_1 \to \widehat{E}_2 \otimes L$ also covering $\vp$, such that: $\Phi_E^* (\nabla^{A_2} \otimes \nabla^L) = \nabla^{A_1}$. 
\end{theorem}

\begin{proof}
    We shall focus on proving Theorem \ref{thm: reconstruction step} in dimensions $n \geq 3$. The proof for dimension $2$ is identical, except for the fact that we do not have the auxiliary bundles $E_i$ and connections $A_i$ to worry about; the last sentence in the statement of the theorem can thus be dropped in this case.
    
    Let us thus suppose that $n \geq 3$. To prove this theorem, fix $p_1 \in W \cap \widehat{M}_1$ and define the set $\Omega_1 \sub \widehat{M}_1$ to be the largest open connected set containing $p_1$ such that for all $x_1 \in \Omega_1$, there is a unique point $x_2 \in \widehat{M}_2$ such that the images $\cK_2((\widehat{\cE}_2)_{x_2})$ and $j_{\cE}(\cK_1(\widehat{\cE}_1)_{x_1})$ coincide, and the map
    \begin{equation}
        \Phi_{x_1} := \cK_2^{-1} \circ j_{\cE} \circ \cK_2 : 
        (\widehat{\cE}_1)_{x_1} \to (\widehat{\cE}_2)_{x_2}
    \end{equation}
    is a unitary isomorphism of the fibers with respect to their induced inner products. Note that the intertwining relation \eqref{intertwining relation on W} implies that $\Omega_1$ contains $W \cap \widehat{M}_1$. Indeed, the restriction of $\Phi$ to $W \cap \widehat{M}_1$ is $j_{\cE}$. For $x_1 \in \Omega_1$, we will denote the corresponding point $x_2$ in $\widehat{M}_2$ by $\vp(x)$.
    
    We now want to show that $\Omega_1 = \widehat{M}_1$, by way of contradiction. If this is not true, then there is $x_1 \in \partial \Omega_1 \cap \widehat{M}_1^{\mathrm{int}}$. The first step is to show that the map $\Phi$ defined above can be extended to the fiber $(\widehat{\cE}_1)_{x_1}$ as a unitary isomorphism. To this end, we have the following technical lemma, which is essentially the same as in \cite[Lemma 4.5]{Valero2025}. We include the proof for completeness, and since it makes use of the MIT boundary conditions for the Green's kernels.

\begin{lemma}\label{lemma: sequence lemma}
    Let $x_1 \in \d \Omega_1 \setm \d \widehat{M}_1$. Then there is a unique point $x_2 \in \widehat{M}^{\mathrm{int}}_2$ such that the images $j_{\cE} \circ \cK_1((\widehat{\cE}_1)_{x_1})$ and $\cK_2((\widehat{\cE}_2)_{x_2})$ coincide, and the map $\Phi_{x_1} : (\widehat{\cE}_1)_{x_1} \to (\widehat{\cE}_2)_{x_2}$ defined above is unitary. Moreover, for any non-zero $v_{x_1} \in (\widehat{\cE}_1)_{x_1}$, there is a unique non-zero $w_{x_2} \in (\widehat{\cE}_2)_{x_2}$ satisfying
        \begin{equation}
            j_{\cE} \circ \sG_1(v_{x_1}) = \sG_2(w_{x_2}), \ \ \ \ \ |v_{x_1}|_{\cE_1} = |w_{x_2}|_{\cE_2},
        \end{equation}
        and for any sequence $v_{p_k} \to v_{x_1}$ with $p_k \to x_1$, we have $\Phi(v_{p_k}) \to w_{x_2}$.
\end{lemma}

\begin{proof}
    Let $p_k$ be a sequence in $\Omega_1$ with $p_k \to x_1$, and let $q_k := \vp(p_k)$. The images $j_{\cE} \circ \cK_1((\widehat{\cE}_1)_{p_k})$ and $\cK_2((\widehat{\cE}_2)_{q_k})$ then coincide, and $\Phi_{p_k}$ is unitary. Since $\widehat{M}_2$ is compact, we may assume by passing to a subsequence that $q_k \to x_2$ for some $x_2 \in \widehat{M}_2$. We want to show that $x_2$ cannot be in $\d \widehat{M}_2$. So let us assume that $x_2 \in \d \widehat{M}_2$. If $x_2 \in \d \widehat{M}_2 \setm \d M_2$, then for large enough $k$, we have that $p_k$ lies in $W \cup U_1$ (here, $U_1$ is the neighbourhood of $\sU$ given by Proposition \ref{prop: extension in analytic gauge}) in which the identifications of our manifolds and bundles made at the end of Section \ref{sec: extension} still hold. Note that $\Phi|_{W \cup\, U} = j_{\cE}$, which covers the identity map (since we are identifying $\widehat{M}_1$ and $\widehat{M}_2$ in this region). Therefore, it follows that $x_2 \notin \d \widehat{M}_2$, since $x_1 \notin \d \widehat{M}_1$, which yields a contradiction. 
    
    So if $x_2 \in \d \widehat{M}_2$, we must therefore have $x_2 \in \d M_2 \setm \sU$, which we recall forms part of $\d \tilde{M}_2$. In this case, since $p_k \in \Omega_1$ for all $k$, we have
    \begin{equation}
        j_{\cE} \circ \cK_1|_{p_k} = \cK_2|_{q_k} \circ \Phi_{p_k}
    \end{equation}
    as maps on $(\widehat{\cE}_1)_{p_k}$. This implies that
    \begin{equation}\label{sequence lemma eqn 1}
        j_{\cE}\circ K_1(x,p_k) = K_2(x,q_k) \circ \Phi_{p_k}
    \end{equation}
    for all $x \in W$. Since $j_{\cE}$ and $\Phi_{p_k}$ are unitary maps, by choosing orthonormal frames near $x_1$ and $x_2$ and taking determinants of \eqref{sequence lemma eqn 1}, we obtain, for all $x \in W$,
    \begin{equation}
        |\det K_1(x,p_k)| = |\det K_2(x,q_k)|.
    \end{equation}
    Taking limits as $k \to \infty$, we have
    \begin{equation}\label{sequence lemma eqn 2}
        |\det K_1(x,x_1)| = |\det K_2(x,x_2)|.
    \end{equation}
    Now, since $x_2 \in \d M_2 \setm \sU \sub \d \tilde{M}_2$, we have $\bP^{-} \tilde{G}_2^{-}(x_2,y) = 0$ for all $y \in \tilde{M}_2^{\mathrm{int}}$, where $\tilde{G}_2^{-}(x,y)$ is the Green's kernel for MIT$^-$ boundary conditions on $\tilde{M}_2$ (see the discussion preceding \eqref{symmetry of MIT Greens kernel}). But, this implies that $\tilde{G}^{-}_2(x_2, x)$, and by \eqref{symmetry of MIT Greens kernel} also $\tilde{G}_2(x,x_2)$, has non-trivial kernel for all $x \in W$. In particular, since $P_i(x_i)$ and $P_i(x)$ are well-defined, \eqref{sequence lemma eqn 2} becomes
    \begin{equation}\label{sequence lemma eqn 3}
        |\det K_1(x,x_1)| = 0
    \end{equation}
    for all $x \in W$. By analyticity of the rescaled Green's kernels, \eqref{sequence lemma eqn 3} must hold for all $x \in \widehat{M}_1^{\mathrm{int}} \setm \{x_1\}$. But this contradicts the asymptotics given by Proposition \ref{prop: rescaled asymptotics}, which imply that
    \begin{equation}
        \lim_{x \to x_1} \frac{|\det K_1(x,x_1)|}{d_1(x,y)^{N(1-n)}} = \frac{c_n^N}{P_i(x_1)^{2N}},
    \end{equation}
    where $N = \rank{\cE_1}$. This yields a contradiction, and thus $x_2 \notin \d \widehat{M}_2$.

    The rest of the lemma follows straightforwardly. Let $v_{x_1} \in (\widehat{\cE}_1)_{x_1} \setm \{0\}$, and choose a sequence of non-zero vectors $v_k \in (\widehat{\cE}_1)_{p_k}$ converging to $v_{x_1}$. Let $w_k := \Phi_{p_k}(v_k)$, so that $|v_k| = |w_k|$. Since $w_k$ is bounded, by passing to a subsequence, we may assume that $w_k \to w_{x_2}$ for some $w_{x_2} \in (\hat{\cE}_2)_{x_2}$. Since $x_2$ is an interior point, the well-definition and uniqueness of $w_{x_2}$ follows from the continuity of $\cK_i$ and the fact that $\cK_i$ is a linear embedding on the fibers.
\end{proof}

The rest of the proof of Theorem \ref{thm: reconstruction step} is a straightforward adaptation of the proofs of the corresponding theorems in \cite{Gabdurakhmanov2025, Valero2025}, for example, to the conformally covariant setting. In the following, we shall therefore content ourselves with an overview of the argument, referring the reader to the aforementioned references, especially \cite{Gabdurakhmanov2025}, for a more detailed exposition.

Let $x_2$ be the point identified in Lemma \ref{lemma: sequence lemma}. Using Lemma \ref{lemma: sequence lemma}, we want to show that a neighbourhood of $x_1$ must be contained in $\Omega_1$. This contradicts the definition of $x_1$, and therefore shows that we must have $\Omega_1 = \widehat{M}_1$. To this end, let $\sE_i := \cK_i(\widehat{\cE}_i)$. Take $v_{x_1} \in (\widehat{\cE}_1)_{x_1} \setm 0$, and let $x_2 \in \widehat{M}^{\mathrm{int}}_2$ and $w_{x_2} \in (\widehat{\cE}_2)_{x_2}$ be the point and vector given by Lemma \ref{lemma: sequence lemma}. Then by construction, we have $j_{\cE}(\cK_1(v_{x_1})) = \cK_2(w_{x_2})$. Let us denote this common value in $H^s(W, \tilde{\cE}|_W)$ by $\zeta$. Since $\cK_i$ is an injective immersion on $\widehat{\cE}_i^0$, we have that $j_{\cE}(\sE_1)$ and $\sE_2$ are submanifolds of $H^s(W, \tilde{\cE}|_W)$ near $\zeta$, whose tangent spaces at $\zeta$ are given by the images of the differentials $D(j_{\cE} \circ \cK_1)|_{v_{x_1}}$ and $D\cK_2|_{w_{x_2}}$, respectively. Since $\cK_i$ is $C^1$, Lemma \ref{lemma: sequence lemma} implies that $T_{\zeta} j_{\cE}(\sE_1)$ and $T_{\zeta}\sE_2$ coincide as subspaces of $H^s(W, \tilde{\cE}|_W)$; we shall denote this subspace by $\sV$.

Now, observe that if $\pi_{\sV} : H^s(W, \tilde{\cE}|_W) \to \sV$ denotes the orthogonal projection, then $\pi_{\sV} \circ \cK_2$ has invertible derivative near $w_{x_2}$. Therefore, the inverse function theorem gives a $C^1$ map $\cR_2 : \sO \to \widehat{\cE}_2$ defined in a neighbourhood $\sO$ of $\pi_{\sV}(\zeta)$. Moreover, on $\sO$ the image $\sE_2$ is the graph of the map
\begin{equation}\label{F_2 local graph definition}
        \cF_2 : \sO \to \sV^{\bot}, \ \ \ \ \ \cF_2(v) := \cK_2(\cR_2(v)) - v.
\end{equation}
A similar argument yields a $C^1$ map $\cR_1 : \sO \to \hat{\cE}_1$ that locally inverts $j_{\cE} \circ \cK_1$ near $\pi_{\sV}(\zeta)$. Defining the map $\cF_1 :\sO \to \sV^{\bot}$ in analogy with \eqref{F_2 local graph definition}, we have that $j_{\cE}(\sE_1)$ is locally the graph of $\cF_1$ near $\zeta$. It is easy to see that the isomorphism $\Phi$ coincides with $\cR_2 \circ \cR_1^{-1}$ on $\cR_1(\sO) \cap \widehat{\cE}_1|_{\Omega_1}$. We have the following lemma, whose proof is essentially the same as that of \cite[Lemma 3.6, 3.7]{Gabdurakhmanov2025}:

\begin{lemma}\label{lemma: graph functions are real-anal}
        $\cR_i : \sO \to \widehat{\cE}_i$ is real-analytic in a neighbourhood of $\pi_{\sV}(\zeta)$ in $\sV$. Moreover, the maps $\cF_1, \cF_2 : \sO \to \sV^{\bot}$ coincide in a neighbourhood of $\Pi_{\sV}(\zeta)$ in $\sV$.
\end{lemma}

It follows from Lemma \ref{lemma: graph functions are real-anal} that there is a neighbourhood of $v_{x_1}$ in $\widehat{\cE}_1$ on which $\cR_2 \circ \cR_1^{-1}$ is real-analytic. Moreover, since $\cF_1 = \cF_2$ near $\pi_{\sV}(\zeta)$ implies that the images $j_{\cE}(\sE_1)$ and $\sE_2$ coincide near $\zeta$. The following lemma states that one can use this information to show that the images of $j_{\cE} \circ \cK_1$ and $\cK_2$ coincide fiber-wise near $x_1$ and $x_2$:

 \begin{lemma}\label{lemma 3.8 gab}
        There is a neighbourhood $\sN_1$ of $x_1$ such that for all $x \in \sN_1$, there exists $y \in \widehat{M}^{\mathrm{int}}_2$ such that the images $j_{\cE} \circ \cK_1((\widehat{\cE}_1)_x)$ and $\cK_2((\widehat{\cE}_2)_y)$ coincide.
    \end{lemma}

The proof of Lemma \ref{lemma 3.8 gab} is again the same as that of \cite[Lemma 3.8]{Gabdurakhmanov2025}. In particular, we note that Lemma \ref{lemma 3.8 gab} implies that $\Phi_x := \cK_2^{-1} \circ j_{\cE} \circ \cK_1$ is a well-defined map from $(\widehat{\cE}_1)_x$ to $(\widehat{\cE}_2)_y$ for all $x \in \sN_1$. Moreover, by Lemma \ref{lemma: graph functions are real-anal}, the map $\Phi_x$ coincides with $\cR_2 \circ \cR_1^{-1}$ on a neighbourhood $\sN_v$ of $v_{x_1}$ in $\widehat{\cE}_1$. Therefore, $\cR_2 \circ \cR_1^{-1}$ is an isometry on the fibers over $\sN_v \cap \widehat{\cE}_1|_{\Omega_1}$, since $\Phi$ is. From this it follows by analytic continuation that $\cR_2 \circ \cR_1^{-1}$ is an isometry on the fibers over $\sN_v$. Since this argument holds for all $v_{x_1}$ in the unit sphere of $(\widehat{\cE}_1)_{x_1}$, it follows that $\Phi_x$ is an isometry of the fibers for all $x \in \sN_1$. This implies that $\sN_1 \subseteq \Omega_1$. But we have assumed that $x_1$ was a boundary point of $\Omega_1$ in $\widehat{M}_1^{\mathrm{int}}$. We thus get a contradiction, whence it follows that $\Omega_1 = \widehat{M}_1$.

We now finish the proof of Theorem \ref{thm: reconstruction step}. The map $\Phi := \cK_2^{-1} \circ j_{\cE} \circ \cK_1$ is a well-defined fiber-preserving map, which, as shown above, agrees locally with the maps $\cR_2 \circ \cR_1^{-1}$ and is therefore a real-analytic unitary isomorphism of bundles $\widehat{\cE}_1 \to \widehat{\cE}_2$. In particular, $\Phi$ covers a real-analytic diffeomorphism $\vp : \widehat{M}_1 \to \widehat{M}_2$. By construction, we have $\Phi|_W = j_{\cE}$ and $\vp|_W = \id$. In particular, over $W$ we have $\vp^* \tilde{g}_1 = \tilde{g}_2$. All that remains is to show that $\vp$ is a conformal equivalence. Note that since $\tilde{g_1}$ and $\tilde{g}_2$ are not real-analytic, this does not simply follow by analytic continuation from $W$ as in \cite{Valero2025}. Instead we use the fact that these are LCRA metrics to prove the following:

\begin{lemma}\label{lemma: vp is conformal equivalence}
    The map $\vp : \widehat{M}_1 \to \widehat{M}_2$ defined above is a conformal equivalence.
\end{lemma}

\begin{proof}
    Let $\sC_1$ be the largest open subset of $\widehat{M}_1$ where $\vp^*\tilde{g}_2$ is conformal to $\tilde{g}_1$. Since $\vp^* \tilde{g}_2 = \tilde{g}_1$ over $W$, it follows that $W \sub \sC_1$. Now suppose that $x_* \in \d \sC_1 \setm \d \widehat{M}_1$. Since the metrics are LCRA, we can choose coordinates around $x_*$ and $\vp(x_*)$ so that $\tilde{g}_i$ is locally represented by $e^{2 \alpha_i} h_i$, where $h_i$ is an analytic metric. Since $\vp$ is a real-analytic map, we can also assume that $\vp$ is represented by a real-analytic function with respect to these coordinates. Then the metric $\vp^* \tilde{g}_2$ is represented by $e^{2 \tilde{\alpha}_2} \tilde{h}_2$, where $\tilde{h}_2$ is analytic. Now define the metrics $\bar{g}_2 := e^{-2\tilde{\alpha}_2} \vp^*\tilde{g}_2$ and $\bar{g}_1 := e^{-2\alpha_1} \tilde{g}_1$, which are both real-analytic in a neighbourhood $\sO$ of $x_*$. Note that $\bar{g}_2 = c \bar{g}_1$ on $\sO \cap \sC_1$ for some smooth non-zero function $c$. Therefore, on $\sO \cap \sC$ we have $\bar{g}_2 \bar{g}_1^{-1} = c \cdot \id$, whence it follows by analytic continuation that $\bar{g}_1$ and $\bar{g}_2$ are conformal in a neighbourhood of $x_*$. This is hence also true for $\tilde{g}_1$ and $\tilde{g}_2$. This contradicts $x_* \in \d \sC_1 \setm \d\widehat{M}_1$, and so we get $\sC_1 = \widehat{M}_1$.
\end{proof}

\begin{remark}
    One can think of the above proof of Lemma \ref{lemma: vp is conformal equivalence} as analytic continuation of the {\em conformal class} of $\tilde{g}_i$ from $W$. Indeed, the conformal class of a metric $g$ on $M$ can be thought of as a section of the bundle $\cC \to M$ of conformal classes of metrics, defined as the quotient of the bundle of metrics $\cM \to M$ by the natural action of $\R^+$. If the base manifold $M$ is LCRA, then by Proposition \ref{prop: LCRA manifolds have analytic structure} there is a canonical analytic structure on its interior, and thus the bundle of conformal classes $\cC$ inherits the structure of an real-analytic bundle, of which LCRA metrics are automatically real-analytic sections, since they have real-analytic conformal representatives. That $\vp^*[\tilde{g}_2] = [\tilde{g}_1]$ on $\widehat{M}_1$ then follows by analytic continuation since this holds on $W$.
\end{remark}
\vspace{5pt}

We now finish the proof of Theorem \ref{main theorem}. Note that we cannot conclude that $\Phi$ preserves the twisted spin connections $\nabla^{A_i}$ on $\widehat{\cE}_i$ as in \cite{Valero2025}, since the spin connection on $\bS_i$ is not real-analytic. Nevertheless we may proceed as follows. Let us pull back all bundles on $\widehat{M}_2$ to $\widehat{M}_1$ via the real-analytic diffeomorphism $\vp$, so that we can work over the same manifold $\widehat{M}_1$. In particular, $\Phi : \widehat{\cE}_1 \to \widehat{\cE}_2$ becomes a vertical isomorphism over the identity on $\widehat{M}_1$. Note that $\bS_2$ need not be globally isomorphic to $\bS_1$, even though it is the spinor bundle of a metric conformal to $\tilde{g}_1$, since $\bS_2$ may in fact correspond to a different choice of spin structure (see Remark \ref{remark: spin structure construction}). In general, we can only say that $\bS_2 \cong \bS_1 \otimes L$, where $L$ is some complex line bundle on $\widehat{M}_1$ that is the complexification of a real line bundle \cite[\S 2.5]{Friedrich2000}. It is easy to see that $L$ is real-analytic in $\widehat{M}^{\mathrm{int}}_1$. Moreover, $L$ is trivial over $W$. Indeed, the identification $j_S : \widehat{\bS}_1|_W \to \widehat{\bS}_2|_W$ induces a trivialization $j_L : L|_{W} \to \C$.

We thus have a real-analytic unitary isomorphism $\Phi : \widehat{\bS}_1 \otimes \widehat{E}_1 \to \widehat{\bS}_1 \otimes \widehat{E}_2 \otimes L$. Taking a partial trace of $\Phi$ over $\bS_1$ yields a unitary isomorphism
\begin{equation}
    \Phi_E := \frac{1}{\rank{\bS_1}} \Tr_{\,\bS_1} \Phi : \widehat{E}_1 \to \widehat{E}_2 \otimes L,
\end{equation}
which by construction satisfies $\Phi_E|_{W} = j_{E} \otimes j_L^{-1}$, identifying $\widehat{E}_1 \otimes \C$ with $\widehat{E}_1$. The same argument as in \cite[\S 4.2]{Valero2025} then proves the following lemma:

\begin{lemma}\label{lemma: flat connection}
    There is a flat connection $\nabla^L$ on $L$ such that $\Phi_E^*(\nabla^{A_2} \otimes \nabla^L) = \nabla^{A_1}$.
\end{lemma}

One can prove Lemma \ref{lemma: flat connection} by first choosing any flat connection $\nabla^{L'}$ on $L$, and then defining a real-analytic $\End(\widehat{E}_2 \otimes L)$-valued $1$-form $a$ on $\widehat{M}_1$ by
\begin{equation}
    a := (\Phi_E^{-1})^* \nabla^{A_1} - \nabla^{A_2} \otimes \nabla^{L'}.
\end{equation}
Since all the connections here are unitary, we can define $b \in \Omega^1(\widehat{M}_1, i\R)$ as a partial trace:
\begin{equation}
    b := \frac{1}{\rank{E}_1} \Tr_{E_1} a.
\end{equation}
One can check that $\nabla^L := \nabla^{L'} - b$ satisfies $\Phi_E^*(\nabla^{A_2} \otimes \nabla^L) = \nabla^{A_1}$ over $W$, which must hold on the rest of $\widehat{M}_1$ by analytic continuation. We refer the reader to \cite[\S 4.2]{Valero2025} for details. This completes the proof of Theorem \ref{thm: reconstruction step}, and therefore also that of Theorem \ref{main theorem}.
\end{proof}

As in \cite{Valero2025}, when the line bundle $L$ above is known to be trivial, such as if the auxiliary bundles are absent or known to be isomorphic, then analytic continuation shows that we have an isomorphism of Clifford module bundles $\Phi_{\bS} : \bS_1 \to \bS_2$. It is known that this then uniquely determines the $\Spin^c$ structures associated to $\bS_1$ and $\bS_2$ \cite{Plymen1986}. In general, however, since we are working with {\em complex} spinor bundles, it is not possible in general to determine the underlying real spin structures, as different spin structures may in fact be isomorphic as $\Spin^c$ structures. Moreover, Example 4.8 in \cite{Valero2025} shows that we cannot in general do away with the line bundle $L$ in the setting of twisted spinor bundles, since in general we cannot distinguish $\bS \otimes (L \otimes E)$ from $(\bS \otimes L) \otimes E$.

\section*{Acknowledgments}

This work has received funding from the European Research Council (ERC) under the European Union's Horizon 2020 research and innovation programme through the grant agreement 862342. The author is also partially supported by the grants CEX2023-001347-S, RED2022-134301-T and PID2022-136795NB-I00, as well as PID2024-156055NA-I00 and CEX2023-001347-S, funded by MICIU/AEI/10.13039/501100011033.

\vspace{10pt}

\appendix

\section{Linear and canonical equivalence of BC maps}\label{sec: appendix L}

In this appendix, we prove Proposition \ref{prop: linear vs canonical equivalence}. Since the statement only concerns the spinor bundles, we shall forget about the auxiliary connection in this section. Let us therefore suppose that $\Theta_{g_1,\sU}$ is linearly equivalent to $\Theta_{g_2,\sU}$ in the sense of Remark \ref{remark: linear equivalence}. That is, we suppose that there exists a vector bundle isomorphism $F : \bS_1|_{\sU} \to \bS_2|_{\sU}$ that intertwines the BC maps:
\begin{equation}\label{appendix L equiv intertwining map}
    F \circ \Theta_{g_1,\sU} \circ F^{-1} = \Theta_{g_2,\sU}.
\end{equation}
In particular, we must have $F(\bW_1^{\pm}) = \bW_2^{\pm}$. As a consequence, we also have
\begin{equation}\label{appendix L equiv intertwining normals}
    F \circ \gamma_1(\nu_1) \circ F^{-1} = \gamma_2(\nu_2).
\end{equation}
Now, note that if \eqref{appendix L equiv intertwining map} holds, then $F$ also intertwines the principal symbols of the BC maps, which are given by Proposition \ref{prop: principal symbol of BC map}. We therefore have for all $\x \in T^*\sU \setm 0$:
\begin{equation}\label{appendix L equiv intertwining of symbols}
    \frac{1}{|\x|_1} F \circ \gamma_1(\x) \circ F^{-1} = \frac{1}{|\x|_2} \gamma_2(\x).
\end{equation}
In particular, by the same argument as in Remark \ref{remark: principal symbol determines conformal class}, we find that the pullback metrics $g_1|_{\sU}$ and $g_2|_{\sU}$ are conformal. Let us say that $g_2|_{\sU} = e^{2u} g_1|_{\sU}$ for some $u \in C^{\8}(\sU)$. Now let $(e^1_a)$ be an orthonormal frame for $g_1|_{\sU}$. Then $(e^2_a)$ defined by $e^2_a := e^{-u} e^1_a$ is an orthonormal frame for $g_2|_{\sU}$. Moreover, the intertwining relation \eqref{appendix L equiv intertwining of symbols} then implies that  
\begin{equation}\label{appendix L equiv intertwining of basis}
    F \circ \gamma_1(e^1_a) \circ F^{-1} = \gamma_2(e^2_a)
\end{equation}
for $a =1, \dots, n-1$. Note that we are identifying Clifford multiplication by a $1$-form with Clifford multiplication by the corresponding vector under the metric isomorphism, as usual.

We now want to define a map of principal $SO(n)$ bundles, $Q : P_{\SO}(g_1)|_{\sU} \to P_{\SO}(g_2)|_{\sU}$ over $\sU$ as follows. First, note that it suffices to specify $Q$ on an orthonormal frame over $\sU$, and extend it to all other frames by $SO(n)$-equivariance. We thus define $Q$ by
\begin{equation}
    Q(e^1_a, \nu_1) := (e^2_a, \nu_2).
\end{equation}
Let $\hat{Q} : \bS_1|_{\sU} \to \bS_2|_{\sU}$ be the induced isomorphism of spinor bundles. By construction, we have
\begin{equation}\label{appendix L equiv conjugation by Q}
    \hat{Q} \circ \gamma_1(e^1_k) \circ \hat{Q}^{-1} = \gamma_2(e^2_k),
\end{equation}
for $k =1, \dots, n$, where $e^i_n := \nu_i$. Therefore, combining \eqref{appendix L equiv intertwining normals}, \eqref{appendix L equiv intertwining of basis}, and \eqref{appendix L equiv conjugation by Q}, we have
\begin{equation}
    F^{-1}\hat{Q} \circ \gamma_1(e_k^1) \circ \hat{Q}^{-1} F = \gamma_1(e_k^1).
\end{equation}
So $\gamma_1(X) \in \End{\bS_1}$ is invariant under conjugation by $F^{-1}\hat{Q}$, for all $X \in TM|_{\sU}$. Since the Clifford algebra is generated multiplicatively by vectors, it follows that $\gamma(\alpha)$ is invariant under conjugation by $F^{-1} \hat{Q}$ for all $\alpha \in \Cl(M)|_{\sU}$. On the other hand, since $\bS$ is an {\em irreducible} module for the Clifford bundle $\Cl(M)$, Schur's lemma implies that $F^{-1} \hat{Q}$ is a non-zero multiple of the identity at every point. We thus obtain that $F = f \hat{Q}$ for some smooth non-vanishing $f \in C^{\8}(\sU)$.

We now want to identify the function $f$ as the conformal scaling factor indicated in Proposition \ref{prop: linear vs canonical equivalence} when $n \geq 3$. Note that from equation \eqref{appendix L equiv intertwining map} we have
\begin{equation}\label{appendix L intertwining map v2}
    f \hat{Q} \circ \Theta_{g_1, \sU} \circ f^{-1} \hat{Q}^{-1} = \Theta_{g_2, \sU}.
\end{equation}
That is, in the notation of \eqref{conformal change of BC map equivalence}, we have $f \Theta_{g_1, \sU} f^{-1} \sim \Theta_{g_2, \sU}$, where $\sim$ here denotes a {\em canonical} equivalence. To identify the function $f$, we want to compute the subprincipal symbol of both sides with respect to a canonical spinor frame induced by a boundary orthonormal frame $(e_k^i)$ as above. Letting $\vartheta_{-1}^i : \bW_i^+ \to \bW_i^{-}$ denote the subprincipal symbol of $\Theta_{g_i, \sU}$, we find
\begin{equation}\label{appendix L subprincipal LHS} 
     \vartheta^1_{-1}(x',\x) - i \d_{\x} \vartheta^1_{0}(x',\x) \cdot \d_{x'} \log{f^{-1}}  = \vartheta^2_{-1}(x',\x),
\end{equation}
where we are identifying the spinor bundles over $\sU$ via the canonical equivalence $\hat{Q}$, and $x'$ denotes coordinates {\em along} the boundary over $\sU$. On the other hand, recall that $g_2|_{\sU} = e^{2u} g_1|_{\sU}$ for some $u \in C^{\8}(\sU)$. It is thus possible to write down an explicit expression for $\vartheta_{-1}^2$ in terms of $\vartheta_{-1}^1$ and $u$. This computation is carried out in \cite{Valero2025}, to which we refer the reader for a more detailed exposition; recall that the full symbol of the MIT BC map is formally identical to that of the chiral BC map, replacing $\gamma(\nu)$ by $\pm i$ as appropriate. We thereby arrive at (c.f. equation (144) in \cite{Valero2025}),

\begin{equation}\label{appendix L equiv subprincipal comparisson}
    \left( \vartheta^2_{-1} \right)^{\fe}(x',\x) = \left(\vartheta^1_{-1}\right)^{\fe}(x',\x) - \frac{i}{|\x|_1^3} \cT^{\a \b}(x') \x_{\a}\x_{\b}
\end{equation}
where $a, \a, \b$ are boundary indices that range over $1,\dots,n-1$, the superscript $\fe$ indicates that we are taking the part that is even under $\x \mapsto -\x$, and the tensor $\cT^{\a \b}$ is given by
\begin{align}
    \cT^{\a \b} &:= -\frac{n-2}{2} g_1^{\a \b} h^{\lambda}_{\ a} (\d_{\lambda} u) \gamma^a  + \left( \frac{n-1}{2} g_1^{\lambda \a} \d_{\lambda} u + \frac{1}{2} g_1^{\sigma \a} h^{\lambda}_{\ b} (h^{-1})^c_{\ \sigma} (\d_{\lambda} u) \gamma^b \gamma_c \right) \gamma^a h^{\beta}_{\ a}. 
\end{align}
Here, $\gamma^a := \gamma_1(e_a^1)$, and $h^{\a}_{\ a}$ relates the boundary orthonormal frame with the coordinate frame: $e^1_{a} := h^{\a}_{\ a} \d_{\a}$. Thus, equating \eqref{appendix L subprincipal LHS} and \eqref{appendix L equiv subprincipal comparisson}, and setting $v := \log{f}$, we obtain
\begin{align}
    -\frac{i}{|\x|^3_1} \cT^{\a \b} \x_{\a} \x_{\b} &= i \d_{\x_{\lambda}} \vartheta_0^1 \cdot \d_{\lambda} v = \frac{i}{|\x|^3}\left( g_1^{\a \b} h^{\lambda}_{\ b} \d_{\lambda} v - g_1^{\alpha \lambda} h^{\b}_{\ b} \d_{\lambda} v \right) \gamma^b \x_{\a} \x_{\b}. \label{appendix L symbol eqn 1}
\end{align}
Since \eqref{appendix L symbol eqn 1} holds for all $\x \in T^*\sU \setm 0$, it follows that the symmetrized coefficients of $\x_{\a} \x_{\b}$ must be equal on both sides. Contracting these tensors with $(g_1)_{\a \b}$ and simplifying, we obtain
\begin{equation}
    (n-2) \gamma(dv) = -(g_1)_{\a \b} \cT^{\a \b} = -\frac{1}{2}(n-2)(n-1) \gamma(du). 
\end{equation}
Since $n > 2$ and $\gamma$ is injective, it follows that
\begin{equation}
    dv = -\frac{n-1}{2} du,
\end{equation}
whence $f = e^{-\frac{n-1}{2} u}$ up to a locally constant function. This proves Proposition \ref{prop: linear vs canonical equivalence}. \hfill \qed

\vspace{20pt}
\addcontentsline{toc}{section}{References}
\printbibliography

\end{document}